\documentclass[11pt]{amsart}

\title[Symmetric bases for FEEC spaces]{Symmetric Bases for Finite Element Exterior Calculus Spaces}
\author{Yakov Berchenko-Kogan}
\address{Department of Mathematical Sciences\\Florida Institute of Technology\\Melborune, FL}
\email{yberchenkokogan@fit.edu}


\usepackage[unicode]{hyperref}
\usepackage{amsthm, amssymb, tikz}
\usetikzlibrary{cd}

\newtheorem{theorem}{Theorem}[section]
\newtheorem{proposition}[theorem]{Proposition}
\newtheorem{corollary}[theorem]{Corollary}
\newtheorem{lemma}[theorem]{Lemma}
\theoremstyle{definition}
\newtheorem{definition}[theorem]{Definition}

\newtheorem{example}[theorem]{Example}
\newtheorem{remark}[theorem]{Remark}

\newcommand{\abs}[1]{\left\lvert{#1}\right\rvert}
\newcommand{\mb}{\mathbf}
\newcommand{\Z}{\mathbb Z}
\newcommand{\R}{\mathbb R}
\newcommand{\cP}{\mathcal P}
\newcommand{\oP}{\mathring{\mathcal P}}
\newcommand{\oV}{\mathring V}
\DeclareMathOperator{\Span}{span}
\DeclareMathOperator{\tr}{tr}
\DeclareMathOperator{\Ind}{Ind}
\DeclareMathOperator{\Res}{Res}

\begin{document}

\begin{abstract}
  In 2006, Arnold, Falk, and Winther developed finite element exterior calculus, using the language of differential forms to generalize the Lagrange, Raviart--Thomas, Brezzi--Douglas--Marini, and N\'ed\'elec finite element spaces for simplicial triangulations. In a recent paper, Licht asks whether, on a single simplex, one can construct bases for these spaces that are invariant with respect to permuting the vertices of the simplex. For scalar fields, standard bases all have this symmetry property, but for vector fields, this question is more complicated: such invariant bases may or may not exist, depending on the polynomial degree of the element.

  In dimensions two and three, Licht constructs such invariant bases for certain values of the polynomial degree $r$, and he conjectures that his list is complete, that is, that no such basis exists for other values of $r$. In this paper, we show that Licht's conjecture is true in dimension two. However, in dimension three, we show that Licht's ideas can be extended to give invariant bases for many more values of $r$; we then show that this new larger list is complete. Along the way, we develop a more general framework for the geometric decomposition ideas of Arnold, Falk, and Winther.
\end{abstract}

\maketitle

\section{Introduction}

As a starting example, consider the space of quadratic real-valued functions on a tetrahedron $T$. Via barycentric coordinates, we can map $T$ onto the standard simplex $T^3\subset\R^4$ consisting of all nonnegative quadruplets $(\lambda_0,\lambda_1,\lambda_2,\lambda_3)$ satisfying $\lambda_0+\lambda_1+\lambda_2+\lambda_3=1$. Here, we have a natural basis for the $10$-dimensional space of quadratic polynomials:
\begin{equation}\label{eq:symmetricbasis}
  \{\lambda_0^2,\lambda_1^2,\lambda_2^2,\lambda_3^2,\lambda_0\lambda_1,\lambda_0\lambda_2,\lambda_0\lambda_3,\lambda_1\lambda_2,\lambda_1\lambda_3,\lambda_2\lambda_3\}.
\end{equation}
We immediately see that if we permute the coordinates $\{\lambda_0,\lambda_1,\lambda_2,\lambda_3\}$, this basis is mapped to itself. Equivalently, this basis is invariant with respect to isometries $T^3\to T^3$ that permute the vertices of the tetrahedron.

This symmetry invariance is convenient. For example, suppose we wanted to compute the matrix of the bilinear map $a(u, v):=\int_{T^3}\nabla u\cdot\nabla v\,d\mathrm{vol}$ with respect to this basis. Naively, we would have to do $100$ computations, or $55$ if we take into account the fact that $a(u, v) = a(v, u)$. However, if we also take into account the permutation symmetry, it suffices to do just $7$ computations:
\begin{multline}\label{eq:7pairs}
  a(\lambda_0^2,\lambda_0^2),\quad a(\lambda_0^2,\lambda_1^2),\quad a(\lambda_0^2,\lambda_0\lambda_1),\quad a(\lambda_0^2,\lambda_1\lambda_2),\\
  a(\lambda_0\lambda_1,\lambda_0\lambda_1),\quad a(\lambda_0\lambda_1,\lambda_0\lambda_2),\quad a(\lambda_0\lambda_1,\lambda_2\lambda_3).
\end{multline}
All other computations can be handled by symmetry. For instance, $a(\lambda_2^2,\lambda_0\lambda_2)=a(\lambda_0^2,\lambda_0\lambda_1)$ via a permutation that sends $\lambda_2$ to $\lambda_0$ and $\lambda_0$ to $\lambda_1$.

In this example, our domain was a regular tetrahedron $T^3$ for simplicity, but we note here that we can exploit symmetry to compute forms such as $a(u, v)=\int_T\nabla u\cdot\nabla v\,d\mathrm{vol}(T)$ even on an arbitrary tetrahedron $T$. A brief sketch of the procedure is as follows: On the reference simplex $T^3$, we compute $A(u, v) := \int_{T^3}du\otimes dv\,d\mathrm{vol}(T^3)=:\sum_{i,j}A_{ij}(u, v)\,d\lambda_i\otimes d\lambda_j$ for the $7$ pairs $(u, v)$ in \eqref{eq:7pairs}; the values of $a(u, v)$ for all $55$ pairs can then be quickly computed from row/column permutations of the $A_{ij}$ and the Jacobian matrix of the affine barycentric coordinate transformation sending $T^3$ to $T$.

For scalar functions, the symmetric basis construction in \eqref{eq:symmetricbasis} works in all dimensions and all polynomial degrees. However, for vector fields, the situation is more complicated. In a recent paper \cite{li19}, Licht attacks the question of when such symmetric bases exist for the Raviart--Thomas elements \cite{rt77}, the Brezzi--Douglas--Marini elements \cite{bdm85}, and the N\'ed\'elec elements \cite{n80, n86}, using the general framework of finite element exterior calculus \cite{a13, afw06, afw09, afw10}. Henceforth in this paper, we will also adopt the framework of finite element exterior calculus, but, for reference, we list the correspondences between the classical finite element spaces and the finite element exterior calculus spaces in Table~\ref{tab:classicalfeec}.

\begin{table}
  \centering
  \begin{tabular}{ll}
    FEEC space&classical counterpart\\\hline
    $\cP_r\Lambda^0(T^n)=\cP_r^-\Lambda^0(T^n)$&Lagrange elements\\
    $\cP_r\Lambda^n(T^n)=\cP_{r+1}^-\Lambda^n(T^n)$&discontinuous elements\\
    $\cP_r\Lambda^1(T^2)$&Brezzi--Douglas--Marini $H(\operatorname{div})$ elements\\
    $\cP_r^-\Lambda^1(T^2)$&Raviart--Thomas $H(\operatorname{div})$ elements\\
    $\cP_r\Lambda^1(T^3)$&N\'ed\'elec $H(\operatorname{curl})$ elements of the second kind\\
    $\cP_r^-\Lambda^1(T^3)$&N\'ed\'elec $H(\operatorname{curl})$ elements of the first kind\\
    $\cP_r\Lambda^2(T^3)$&N\'ed\'elec $H(\operatorname{div})$ elements of the second kind\\
    $\cP_r^-\Lambda^2(T^3)$&N\'ed\'elec $H(\operatorname{div})$ elements of the first kind\\
  \end{tabular}
  \vspace\baselineskip
  \caption{A summary of \cite[Tables 5.1 and 5.2]{afw06} listing the correspondences between the finite element exterior calculus spaces and the classical finite element spaces; here $T^n$ denotes the $n$-dimensional simplex. See also \cite[Tables 2.1 and 2.2]{afw06}.}
  \label{tab:classicalfeec}
\end{table}

In the above example with scalar fields, we saw that our basis of $\cP_2\Lambda^0(T^3)$ was invariant with respect to the permutation group $S_4$ that permutes the vertices of the simplex. However, in both Licht's paper \cite{li19} and our paper, we want to weaken the notion of invariance to allow basis elements to map to other basis elements \emph{up to sign}. We can motivate this definition by considering the space $\cP_2\Lambda^3(T^3)$ consisting of densities with quadratic coefficients. Initially, it may seem that the situation is no different from the scalar field situation; we can just use the same basis multiplied by the volume form:
\begin{multline*}
  \{\lambda_0^2\,d\mathrm{vol},\lambda_1^2\,d\mathrm{vol},\lambda_2^2\,d\mathrm{vol},\lambda_3^2\,d\mathrm{vol},\\
  \lambda_0\lambda_1\,d\mathrm{vol},\lambda_0\lambda_2\,d\mathrm{vol},\lambda_0\lambda_3\,d\mathrm{vol},\lambda_1\lambda_2\,d\mathrm{vol},\lambda_1\lambda_3\,d\mathrm{vol},\lambda_2\lambda_3\,d\mathrm{vol}\}.
\end{multline*}
However, whereas the volume form is invariant under orientation-preserving isometries of $T^3$, it will change sign if we apply an orientation-reversing isometry. Thus, if we permute the vertices of $T^3$ with an odd permutation, we will send $d\mathrm{vol}$ to $-d\mathrm{vol}$ and, therefore, send each basis element to the \emph{negative} of another basis element.

This issue is pervasive: Even if we were to try to restrict to orientation-preserving isometries of $T^3$, such isometries can still reverse the orientation of edges, and so we will have similar issues with bases of one-forms. On the other hand, there is not much downside if basis elements can map to negatives of other basis elements: We must keep track of signs, but otherwise we have the same kind of advantages as in our example of computing the matrix of $a(u, v)=\int_{T^3}\nabla u\cdot\nabla v\,d\mathrm{vol}$ above. Thus, we establish the following definition.

\begin{definition}\label{def:invariantuptosign}
  Let $V$ be a $G$-representation, that is, a (real) vector space with an action of a group $G$. We say that a basis $\xi_1,\dotsc,\xi_n$ is a \emph{$G$-invariant basis up to sign} if for every $g\in G$ and every $i$, we have that $g\xi_i=\pm\xi_{i'}$ for some $i'$.
\end{definition}

\begin{remark}
  Invariance up to sign is the only notion of basis invariance discussed in this paper, so we will often simply write \emph{$G$-invariant basis}. The group $G$ will also be omitted if it is clear from context. We also remark that Licht \cite{li19} uses different terminology, using the term $\R$-invariant basis to denote a basis that is invariant up to sign, in order to distinguish from $\mathbb C$-invariance, which we do not discuss in the present paper.
\end{remark}

\begin{remark}
  Representations that have a $G$-invariant basis up to sign are sometimes referred to as \emph{monomial representations}.
\end{remark}

Another motivation comes from Whitney forms \cite{wh57}. In our notation, the \emph{Whitney forms} are the space $\cP_1^-\Lambda^k(T^n)$, the lowest-order incarnation of the $\cP^-_r\Lambda^k(T^n)$ spaces. The main property of Whitney $k$-forms is that they are uniquely determined by their integrals on $k$-dimensional faces. As such, they have a natural basis, the \emph{Whitney elementary forms}: a Whitney elementary form has integral one on a particular $k$-dimensional face $F$ of $T^n$ and integral zero on all the others. One can check that this basis is invariant, but only up to sign, due to the fact that isometries of $T^n$ take $k$-dimensional faces to $k$-dimensional faces, but may reverse their orientations. From this perspective, our question is: Does this symmetry property continue to hold for the higher-order analogues $\cP_r^-\Lambda^k(T^n)$ of the Whitney forms? What about the other family, the $\cP_r\Lambda^k(T^n)$ spaces? The Whitney forms enjoy many nice geometric properties \cite{hi03}, such as a geometric interpretation of the exterior derivative; the symmetry property addressed in this paper appears necessary for those geometric properties to persist in the generalization to the higher-order forms of finite element exterior calculus.

\subsection*{Main results} Each of the finite element spaces $\cP_r\Lambda^k(T^n)$ and $\cP_r^-\Lambda^k(T^n)$ has an action of the symmetric group $S_{n+1}$ that permutes the $n+1$ vertices of the simplex $T^n$. Our question then is: Which of these spaces have invariant bases up to sign, and which do not? Licht \cite{li19} gives a partial answer to this question in dimensions two and three, constructing invariant bases for certain values of $r$ (depending on the family). In this paper, we develop a more general theory that lets us give a complete answer in dimensions two and three (and could, with some computation, be applied to higher dimensions as well). In two dimensions, we show that, for the spaces where Licht did not find invariant bases, there are indeed no such bases, thereby showing that Licht's list in complete. On the other hand, in three dimensions, we find many additional spaces that have invariant bases; we then show that our new expanded list is complete.

\subsection*{Techniques} In both Licht's work \cite{li19} and in the current paper, we induct using the \emph{geometric decomposition} and \emph{duality} properties of the finite element exterior calculus spaces. The geometric decomposition allows us to decompose finite element spaces on $T^n$ in terms of finite element spaces on the faces $F$ of $T^n$ that vanish on the boundary $\partial F$. Meanwhile, duality gives an isomorphism between spaces of forms that vanish on the boundary and spaces of forms of lower degree (with no boundary conditions). Together, these two facts allow us to inductively construct invariant bases. The main innovations in this paper that allow us to prove the stronger results are the following.
\begin{itemize}
\item We give a representation-theoretic obstruction to the existence of invariant bases.
\item We develop a more thorough understanding of the representation theory of the geometric decomposition. 
\item In three dimensions, we expand the base case of the inductive construction, allowing us to construct invariant bases for more spaces.
\end{itemize}

\subsection*{Outline} Section~\ref{sec:preliminaries} is devoted to preliminaries: We define the finite element exterior calculus spaces and how the symmetric group acts on them, and we introduce the results we will need from representation theory. In Section~\ref{sec:geodecomp}, we develop a generalization of the geometric decomposition theory of Arnold, Falk, and Winther \cite{afw09} that is well-suited for easily understanding the representation theory of these spaces. We relegate most of the discussion of the correspondence between these two notions of geometric decomposition to Appendix~\ref{sec:extension}, as it does not impact our results. The results themselves are in Section~\ref{sec:results}.

\section{Preliminaries}\label{sec:preliminaries}
This preliminaries section is split into four subsections. The first three are devoted to representation theory, with the general theory in Subsection~\ref{subsec:reptheory}, the theory of induced representations in Subsection~\ref{subsec:induced}, and relevant examples in Subsection~\ref{subsec:repexamples}, relying on the references \cite{cure62, fuha91}; see also \cite{se77}. Meanwhile, in Subsection~\ref{subsec:feec}, we briefly introduce the finite element spaces $\cP_r\Lambda^k(T^n)$ and $\cP_r^-\Lambda^k(T^n)$ and discuss how the symmetric group $S_{n+1}$ acts on them; this subsection culminates with the duality isomorphism that, along with the geometric decomposition, is the main ingredient of the inductive basis construction.

\subsection{Representation theory}\label{subsec:reptheory}
While there are many equivalent perspectives on representation theory, in this paper we will view a representation as a vector space endowed with additional structure. We briefly review the key definitions and theorems we will use in this paper. For this paper, all vector spaces are over the real numbers, though the results in this section hold more generally for fields of characteristic zero.

\begin{definition}
  Given a group $G$, a \emph{$G$-representation} or simply \emph{representation} $V$ is a vector space endowed with a product $G\times V\to V$ denoted $(g, v)\mapsto gv$, such that
  \begin{itemize}
  \item for all $g\in G$ the induced map $V\to V$ defined by $v\mapsto gv$ is linear, and
  \item $g(hv)=(gh)v$ for all $g,h\in G$ and $v\in V$.
  \end{itemize}
\end{definition}

\begin{definition}\label{def:trivial}
  Any group $G$ has a trivial one-dimensional representation, which we denote $\mb1$. The action is $gv=v$.
\end{definition}

\begin{proposition}\label{prop:1invariant}
  The trivial representation $\mb1$ has an $G$-invariant basis.
\end{proposition}

\begin{proof}
  Let $v$ be a nonzero element of $\mb1$. Then $\{v\}$ is an invariant basis.
\end{proof}

Representations have a stronger notion of isomorphism than vector spaces do.

\begin{definition}
  Given two representations $V$ and $W$, a linear map $f\colon V\to W$ is \emph{equivariant} if $f(gv)=g(f(v))$. If, moreover, $f$ is a bijection, then we say that $V$ and $W$ are \emph{isomorphic} representations.
\end{definition}

\begin{definition}
  If $V$ is a $G$-representation and $W$ is a linear subspace of $V$, then $W$ is an \emph{invariant subspace} or \emph{subrepresentation} if $gW=W$ for all $g\in G$. If so, $W$ itself has the structure of a $G$-representation.

  If the only subrepresentations of $V$ are $0$ and $V$, then $V$ is \emph{irreducible}.
\end{definition}

Given two $G$-representations $V$ and $W$, the vector spaces $V\oplus W$ and $V\otimes W$ inherit a canonical action of $G$.
\begin{definition}
  Let $V$ and $W$ be two $G$-representations.

  The \emph{direct sum representation}, denoted $V\oplus W$, is defined on the vector space $V\oplus W$ by $g(v,w):=(gv,gw)$ for all elements $(v,w)\in V\oplus W$. For a nonnegative integer $n$, we will let $nV$ denote the direct sum of $n$ copies of $V$.

  The \emph{tensor product representation}, denoted $V\otimes W$, is defined on the vector space $V\otimes W$ by $g(v\otimes w)=gv\otimes gw$ for all $v\in V$ and $w\in W$.
\end{definition}

We observe here that these operations preserve the property of having invariant bases.
\begin{proposition}\label{prop:sumtensor}
  Let $V$ and $W$ be $G$-representations. If $V$ and $W$ have $G$-invariant bases up to sign, then so do $V\oplus W$ and $V\otimes W$.
\end{proposition}

\begin{proof}
  Let $\xi_1,\dotsc,\xi_m$ and $\eta_1,\dotsc,\eta_n$ be invariant bases for $V$ and $W$, respectively. Then one can see that $\xi_1,\dotsc,\xi_m,\eta_1,\dotsc,\eta_n$ is an invariant basis for $V\oplus W$. Likewise, $\{\xi_i\otimes\eta_j\}$ is an invariant basis of $V\otimes W$. Indeed, we know that $g\xi_i=\pm\xi_{i'}$ for some $i'$ and $g\eta_j=\pm\eta_{j'}$ for some $j'$, so
  \begin{equation*}
    g(\xi_i\otimes\eta_j)=(g\xi_i)\otimes(g\eta_j)=(\pm\xi_{i'})\otimes(\pm\eta_{j'})=\pm(\xi_{i'}\otimes\eta_{j'}),
  \end{equation*}
  as desired.
\end{proof}

A key result of the representation theory of finite (or, more generally, compact) groups is the unique decomposition into irreducibles.

\begin{theorem}[Maschke]
  If $G$ is a finite group, then any finite-dimensional representation $V$ decomposes as a direct sum of irreducible subrepresentations.
  \begin{equation*}
    V\cong V_1\oplus\dotsb\oplus V_k.
  \end{equation*}
\end{theorem}

Since every finite group has finitely many irreducible representations up to isomorphism, Maschke's theorem yields a powerful method for assessing which representations satisfy a representation-theoretic property such as having an invariant basis:
\begin{itemize}
\item Identify all of the irreducible representations of $G$.
\item For each possible direct sum of these representations, determine whether the property holds.
\item For a given representation $V$, compute its decomposition into irreducibles. Based on this decomposition and the results of the previous step, conclude whether the property holds.
\end{itemize}

There is a straightforward way of identifying the irreducible representations of $G$ using the regular representation.

\begin{definition}\label{def:regular}
  Any finite group $G$ has a representation called the \emph{regular representation}, denoted $\R[G]$. As a vector space, $\R[G]$ is the space of formal linear combinations $\sum_{g_i\in G}a_ig_i$ of elements of $G$. The group action is
  \begin{equation*}
    g\sum_{g_i\in G}a_ig_i:=\sum_{g_i\in G}a_i(gg_i).
  \end{equation*}
\end{definition}

An important property of the regular representation is that it contains all irreducible representations.
\begin{proposition}\label{prop:irreducibleregular}
  Any irreducible representation of a group $G$ is a subrepresentation of the regular representation of $G$.
\end{proposition}
This allows us to identify all irreducible representations of a group $G$ by decomposing the regular representation into irreducibles. All irreducible representations of $G$ will be present in this decomposition (generally with multiplicity).

\begin{definition}
  Given a $G$-representation $V$ and a subgroup $H$ of $G$, we can also view $V$ as an $H$-representation by ``forgetting'' the additional structure of multiplication by elements of $g$ that are not in $H$. This representation is called the \emph{restriction} of the $G$-representation $V$ to $H$. When there is no risk of ambiguity, we will refer to this restricted $H$-representation simply as $V$. When we must disambiguate, we will refer to it as $\Res^G_H(V)$.
\end{definition}

Less trivially, given an $H$-representation $W$, we can construct a canonical $G$-representation $V$; this construction is the subject of the next subsection.

\subsection{Induced representations}\label{subsec:induced}
It turns out that we can characterize the $G$-representations that admit $G$-invariant bases up to sign using the theory of \emph{induced representations}. We give a brief introduction, largely following \cite[Subsection 3.3]{fuha91}.

\begin{definition}\label{def:induce}
  Let $V$ be a $G$-representation. Let $W$ be a subspace of $V$ that is invariant under the action of a subgroup $H$ of $G$. Then the translate $gW$ depends only on the left coset $gH\in G/H$, so it makes sense to write $\sigma W$ for $\sigma\in G/H$. If it so happens that $V$ is the direct sum of these translates, that is
  \begin{equation*}
    V=\bigoplus_{\sigma\in G/H}\sigma W,
  \end{equation*}
  then we say that $V$ is \emph{induced} by $W$ and write $V=\Ind_H^GW$.
\end{definition}

\begin{proposition}[{\cite[Subsection 3.3]{fuha91}}]
  The induced representation is unique. That is, the isomorphism class of $V=\Ind_H^GW$ as a $G$-representation depends only on the isomorphism class of $W$ as an $H$-representation.
\end{proposition}

\begin{proposition}[{\cite[Example 3.15]{fuha91}}]\label{prop:inducesum}
  Induction respects direct sums, that is,
  \begin{equation*}
    \Ind_H^G\bigoplus W_i\cong\bigoplus\Ind_H^GW_i.
  \end{equation*}
\end{proposition}

\begin{proposition}[{\cite[Exercise 3.16(b)]{fuha91}}]\label{prop:inducetrans}
  Induction is transitive, that is, if $H$ is a subgroup of $K$, which is a subgroup of $G$, and $W$ is an $H$-representation, then
  \begin{equation*}
    \Ind_H^GW\cong\Ind_K^G\Ind_H^KW.
  \end{equation*}
\end{proposition}

\begin{proposition}[{\cite[Section 12D]{cure62}}]\label{prop:induceregular}
  Induction sends regular representations to regular representations.
  \begin{equation*}
    \Ind_H^G\R[H]\cong\R[G].
  \end{equation*}
\end{proposition}

In particular, if $H$ is the trivial subgroup, we have the following.
\begin{corollary}\label{cor:inducetrivial}
  Let $e$ be the trivial subgroup of $G$, and let $\mb 1=\R[e]$ be the trivial $e$-representation. Then
  \begin{equation*}
    \Ind_e^G\mb 1\cong\R[G].
  \end{equation*}
\end{corollary}

Finally, we have the following property relating induction, restriction, and the tensor product.
\begin{proposition}[{\cite[Exercise 3.16(a)]{fuha91}}]\label{prop:inducetensor}
  Let $H$ be a subgroup of a group $G$, let $V$ be a $G$-representation, and let $W$ be an $H$-representation. Then
  \begin{equation*}
    V\otimes\Ind_H^GW\cong\Ind_H^G\left(\Res_H^G(V)\otimes W\right).
  \end{equation*}
\end{proposition}

\begin{corollary}\label{cor:tensorregular}
  Tensoring a representation $V$ with the regular representation yields the direct sum of $\dim V$ copies of the regular representation.
  \begin{equation*}
    V\otimes\R[G]\cong(\dim V)\,\R[G].
  \end{equation*}
\end{corollary}

\begin{proof}
  In Proposition~\ref{prop:inducetensor}, let $H$ be the trivial group $e$ and let $W$ be the trivial $e$-representation $\mb 1$. Then the left-hand side of Proposition~\ref{prop:inducetensor} yields
  \begin{equation*}
    V\otimes\Ind_e^G\mb 1\cong V\otimes\R[G]
  \end{equation*}
  by Corollary~\ref{cor:inducetrivial}. Meanwhile, to compute the right-hand side, first observe that if we view $V$ as an $e$-representation, it is simply a vector space with no additional structure, and therefore its decomposition as an $e$-representation is
  \begin{equation*}
    \Res_e^G(V)\cong\mb 1\oplus\dotsb\oplus\mb 1=(\dim V)\,\mb 1.
  \end{equation*}
  Thus, using Proposition~\ref{prop:inducesum} and Corollary~\ref{cor:inducetrivial}, we have
  \begin{multline*}
    \Ind_e^G\left(\Res_e^G(V)\otimes\mb 1\right)\cong\Ind_e^G\left(\Res_e^G(V)\right)\cong\Ind_e^G\left((\dim V)\,\mb 1\right)\\
    \cong(\dim V)\,\Ind_e^G\mb 1\cong(\dim V)\,\R[G].\qedhere
  \end{multline*}

\end{proof}

We can rephrase Definition~\ref{def:invariantuptosign} in terms of induced representations. The intuition is that a basis is equivalent to a decomposition of a vector space into a direct sum of lines, that is, one-dimensional subspaces. With this perspective, a basis being invariant up to sign means that the group action permutes these lines. Focusing on a single orbit of this action, the direct sum of the lines in this orbit is precisely the definition of a representation induced by one of the lines. Summing over the orbits gives the following characterization.

\begin{proposition}[{\cite[page 347]{cure62}}]\label{prop:invariantinduced}
  Let $G$ be a finite group. A $G$-representation $V$ has a $G$-invariant basis up to sign if and only if it is the direct sum of $G$-representations that are induced by one-dimensional representations, that is,
  \begin{equation*}
    V=\bigoplus_i\Ind_{H_i}^GL_i,
  \end{equation*}
  where each $H_i$ is a subgroup of $G$ and each $L_i$ is a one-dimensional $H_i$-representation.
\end{proposition}

\begin{proof}
  Assume that $V$ is a direct sum of $G$-representations that are induced by one-dimensional representations. By Proposition~\ref{prop:sumtensor}, it suffices to show that an individual summand $\Ind_H^GL$ has an invariant basis up to sign. By definition,
  \begin{equation*}
    \Ind_H^GL=\bigoplus_{\sigma\in G/H}\sigma L,
  \end{equation*}
  a sum of one-dimensional subspaces. Thus, if $\xi$ is a nonzero element of the line $L$, then $\{g_\sigma\xi\}_{\sigma\in G/H}$ is a basis for $\Ind_H^GL$, where $g_\sigma\in G$ is a fixed representative of the coset $\sigma$. We claim that this basis is invariant up to sign. First, since $L$ is an $H$-invariant subspace, for $h\in H$, we have that $h\xi$ is in the line $L$, so $h\xi=\alpha\xi$ for some real number $\alpha$. Since $h^{\abs H}$ is the identity, $\alpha^{\abs H}=1$, so $\alpha=\pm1$. Next, for $g\in G$, we have $g=g_\sigma h$ for some $\sigma\in G/H$ and $h\in H$, so $g\xi=g_\sigma h\xi=g_\sigma(\pm\xi)=\pm g_\sigma\xi$. Finally, applying this result to $gg_\sigma\in G$, we see that $g(g_\sigma\xi)=(gg_\sigma)\xi=\pm g_{\sigma'}\xi$ for some $\sigma'\in G/H$, proving our claim.

  Conversely, assume that $V$ has a $G$-invariant basis $\{\xi_j\}_{j=1}^{\dim V}$ up to sign. Let $l_j$ be the line spanned by $\xi_j$. Since any $g\in G$ sends each $\xi_j$ to $\pm\xi_{j'}$ for some $j'$, we see that the group $G$ permutes the set of lines $\{l_j\}$. This action partitions the set of lines into orbits. For each orbit $i$, let $V_i$ denote the span of the lines in the orbit, and so we have $V=\bigoplus_iV_i$.

  Focusing on a single orbit, let $J_i$ denote set of all $j$ such that $l_j$ is in the orbit $i$. Let $L_i$ be an arbitrarily chosen line among the $l_j$ in this orbit. Let the subgroup $H_i$ be the stabilizer of $L_i$ with respect to the action of $G$. By the orbit-stabilizer theorem, for every $j\in J_i$ we have $l_j=\sigma L_i$ for a unique $\sigma\in G/H_i$. Therefore, $V_i=\bigoplus_{\sigma\in G/H_i}\sigma L_i$, so $V_i=\Ind_{H_i}^GL_i$ by definition.
\end{proof}

Together, Propositions~\ref{prop:inducesum}, \ref{prop:inducetrans} and \ref{prop:invariantinduced} show that induction preserves the property of having invariant bases.

\begin{proposition}\label{prop:induce}
  Let $K$ be a subgroup of a finite group $G$, and let $V=\Ind_K^GW$. If $W$ has a $K$-invariant basis up to sign, then $V$ has a $G$-invariant basis up to sign.
\end{proposition}

\begin{proof}
  \begin{equation*}
    V=\Ind_K^GW=\Ind_K^G\bigoplus_i\Ind_{H_i}^KL_i\cong\bigoplus_i\Ind_K^G\Ind_{H_i}^KL_i\cong\bigoplus_i\Ind_{H_i}^GL_i.\qedhere
  \end{equation*}
\end{proof}

\subsection{Examples of representations}\label{subsec:repexamples}

In this subsection, we will give examples of representations that will be relevant in this paper. We will often denote representations by a bold number representing its dimension, such as $\mb2$ for a two-dimensional representation. Of course, occasionally there can be non-isomorphic representations of the same dimension that are both relevant, in which case we will disambiguate.

\subsubsection{Representations of the symmetric group}
\begin{definition}
  We denote the permutation group by $S_{n+1}$ and view it as the group of permutations of the set $\{0,\dotsc,n\}$.
\end{definition}

\begin{definition}\label{def:1p}
  The symmetric group $S_{n+1}$ has a nontrivial one-dimensional irreducible representation, which we denote $\mb1'$, called the \emph{sign representation}. The action is $\pi v=v$ if $\pi\in S_{n+1}$ is an even permutation and $\pi v=-v$ if $\pi$ is an odd permutation.
\end{definition}

\begin{proposition}
  The sign representation $\mb1'$ has an $S_{n+1}$-invariant basis up to sign.
\end{proposition}

\begin{proof}
  Let $v$ be a nonzero element of $\mb1'$. Then $\{v\}$ is a basis and $\pi v = \pm v$ depending on whether $\pi$ is even or odd.
\end{proof}

\begin{example}\label{eg:tensor1p}
  Let $V$ be any $S_{n+1}$-representation, and let $W=V\otimes \mb1'$. Since $\mb1'$ is one-dimensional, we have that $V$ and $W$ are isomorphic as vector spaces via the map $v\mapsto v\otimes\xi$, where $\xi$ is a fixed nonzero vector in $\mb1'$. However, this map is not $S_{n+1}$-equivariant. Indeed, $\pi(v\otimes\xi)=(\pi v)\otimes\xi$ only if $\pi$ is an even permutation; if $\pi$ is odd then $\pi(v\otimes\xi)=-(\pi v)\otimes\xi$.

  In short, writing $W\cong V\otimes\mb1'$ is a formal way of capturing the idea that $W$ is the same vector space as $V$ with the same action of the permutation group, except that whenever we apply an odd permutation we multiply the resulting vector by $-1$. In general, $V$ and $W$ may or may not be isomorphic as $S_{n+1}$-representations.

  For example, consider $\R^3$ with coordinates $(x,y,z)$, where the group $S_3$ permutes the axes. Let $V$ be the set of constant one-forms, so it is spanned by $dx$, $dy$, and $dz$. Let $W$ be the set of constant $2$-forms, so it is spanned by $dy\wedge dz$, $dz\wedge dx$, and $dx\wedge dy$. The Hodge star yields an isomorphism between $V$ and $W$ as vector spaces, but it is not $S_3$-equivariant. Indeed, consider the permutation that fixes the $z$-axis and swaps the $x$-axis with the $y$-axis. This permutation sends $dz$ to $dz$ but it sends $dx\wedge dy$ to $dy\wedge dx=-dx\wedge dy$. To get an $S_3$-equivariant map, we must instead use the Hodge star to map $W$ to $V\otimes\mb 1'$, sending, for instance, $dx\wedge dy$ to $dz\otimes\xi$. Then, we see that the above permutation sends $dz\otimes\xi$ to $dz\otimes(-\xi)=-dz\otimes\xi$, as needed for equivariance. See also Remark~\ref{rem:cross}.
\end{example}

We remark here that $V$ has an invariant basis up to sign if and only if $W$ does.
\begin{proposition}\label{prop:tensorsign}
  Let $V$ be an $S_{n+1}$-representation. Then $V$ has an invariant basis up to sign if and only if $V\otimes\mb1'$ does.
\end{proposition}

\begin{proof}
  By Proposition~\ref{prop:sumtensor}, if $V$ has an invariant basis, then so does $V\otimes\mb1'$. The converse follows from the fact that
  \begin{equation*}
    (V\otimes\mb1')\otimes\mb1'\cong V\otimes(\mb1'\otimes\mb1')\cong V\otimes\mb1\cong V.\qedhere
  \end{equation*}
\end{proof}

\begin{definition}\label{def:permuterepresent}
  The symmetric group $S_{n+1}$ has an $(n+1)$-dimensional representation, which we denote $\mb{(n+1)}$, defined by permuting the basis vectors. That is, we let $e_0,\dotsc,e_n$ denote a basis of $\R^{n+1}$, and then we let $\pi e_i = e_{\pi(i)}$ and extend linearly.
\end{definition}

By construction, the basis $e_0,\dotsc,e_n$ is invariant.
\begin{proposition}\label{prop:n+1invariant}
  The representation $\mb{(n+1)}$ has an $S_{n+1}$-invariant basis.
\end{proposition}

For the representation $\mb{(n+1)}$, for any permutation $\pi\in S_{n+1}$, we have that $\pi(e_0+\dotsb+e_n)=e_0+\dotsb+e_n$. Hence, the line spanned by the vector $e_0+\dotsb+e_n$ is an invariant subspace of $\mb{(n+1)}$ and is isomorphic to $\mb1$. Since $S_{n+1}$ acts by rotations or reflections, the orthogonal complement of this line is also an invariant subspace.

\begin{definition}\label{def:standard}
  The \emph{standard representation} of the symmetric group $S_{n+1}$ is the $n$-dimensional representation that is the orthogonal complement of the line spanned by the vector $\langle1,1,\dotsc,1\rangle$ in $\mb{(n+1)}$. We will let $\mb n$ denote this representation for $n\ge2$.
\end{definition}

\begin{remark}
  We excluded the case $n=1$ from the above definition of the notation $\mb n$ to avoid a notational conflict, due to the fact that the standard representation of the group $S_2$ is the sign representation $\mb1'$, not the trivial representation $\mb1$.
\end{remark}

\begin{example}\label{eg:equilateral}
  One often introduces the symmetric group $S_3$ as the group of rotations and reflections of an equilateral triangle in the plane, or, equivalently, as the group of rotations and reflections of the plane that permute the vertices of the triangle. The plane with this action of $S_3$ is precisely the standard representation $\mb2$. We can see this more generally as follows.

  Consider the points $e_i-\frac1{n+1}(e_0+\dotsb+e_n)$ for $0\le i\le n$. Observe that they form the vertices of an equilateral simplex in $\mb n$, and observe that the action of $S_{n+1}$ on $\mb n$ permutes these vertices, sending $e_i-\frac1{n+1}(e_0+\dotsb+e_n)$ to $e_{\pi(i)}-\frac1{n+1}(e_0+\dotsb+e_n)$.
\end{example}

A well-known fact is that $\mb n$ is irreducible.

\begin{proposition}\label{prop:np1np1}
  The standard representation $\mb n$ of $S_{n+1}$ is irreducible. Consequently, the irreducible decomposition of $\mb{(n+1)}$ is
  \begin{equation*}
    \mb{(n+1)}\cong\mb n\oplus\mb1.
  \end{equation*}
\end{proposition}

On the surface, it appears that there is some ambiguity of notation. On the one hand, the symbol $\mb 3$ denotes the representation of $S_3$ that permutes the basis vectors of $\R^3$. On the other hand, $\mb3$ denotes the standard representation of $S_4$; concerningly, since $S_3$ is a subgroup of $S_4$, we can restrict the action of $S_4$ to $S_3$ to obtain a three-dimensional representation of $S_3$. Fortunately for our choice of notation, these two representations of $S_3$ are isomorphic.

\begin{proposition}
  Let $V$ be the standard representation of $S_{n+1}$. Then the subgroup $S_n$ of $S_{n+1}$ acts by permuting a basis of $V$.
\end{proposition}

\begin{proof}
  As in Definitions~\ref{def:permuterepresent} and \ref{def:standard}, we take $e_0,\dotsc,e_n$ to be the basis of $\R^{n+1}$ and view $V$ as the subspace that is orthogonal to $e_0+\dotsb+e_n$. We can check that $\{e_i-e_n\mid 0\le i\le n-1\}$ is a basis for $V$. The subgroup $S_n$ of $S_{n+1}$ consists of those permutations $\pi$ of $\{0,\dotsc,n\}$ such that $\pi(n)=n$. Thus, for $\pi\in S_n$, we have $\pi(e_i-e_n)=e_{\pi(i)}-e_n$, as desired.
\end{proof}
  
\subsubsection{Representations of $\Z/3$}\label{subsubsec:z3}
We will make use of the following three representations of $\Z/3$.
\begin{itemize}
\item Per Definition~\ref{def:trivial}, $\Z/3$ has the trivial representation $\mb 1$.
\item Per Definition~\ref{def:permuterepresent}, the group $S_3$ has a representation $\mb 3$. Viewing the group $\Z/3$ as the subgroup of $S_3$ consisting of cyclic permutations of $\{0,1,2\}$, we can restrict this $S_3$-representation to a $\Z/3$-representation. Abusing notation slightly, we will also refer to this $\Z/3$-representation as $\mb3$. Explicitly, letting $g$ be the generator of $\Z/3$ that sends $0$ to $1$, $1$ to $2$, and $2$ to $0$, we have that $g$ acts on $\mb3$ via $ge_0=e_1$, $ge_1=e_2$, and $ge_2=e_0$.
\item Likewise, per Definition~\ref{def:standard}, the group $S_3$ has the standard representation $\mb 2$, which we can restrict to a $\Z/3$-representation that we also denote by $\mb2$. Explicitly, the group $\Z/3$ acts on $\mb2$ by $120^\circ$ rotations. Indeed, $v=e_0-e_1\in\mb2$ per Definition~\ref{def:standard}, and the generator $g$ of $\Z/3$ sends $v$ to $w=e_1-e_2$. One can check that the angle between $v$ and $w$ is $120^\circ$. As a consequence, we have that $\mb2$ is irreducible as a $\Z/3$-representation.
\end{itemize}

Per the above discussion, Proposition~\ref{prop:np1np1} still holds in the context of $\Z/3$-representations.
\begin{proposition}
  As $\Z/3$-representations, the irreducible decomposition of $\mb 3$ is $\mb3\cong\mb 2\oplus\mb1$.  
\end{proposition}

We also have the following easy observations.
\begin{proposition}\label{prop:13invariant}
  $\mb 1$ and $\mb 3$ have $\Z/3$-invariant bases.
\end{proposition}

\begin{proof}
  If a basis is $S_3$-invariant, then it is also $\Z/3$-invariant, so $\mb1$ and $\mb3$ have $\Z/3$-invariant bases by Propositions~\ref{prop:1invariant} and \ref{prop:n+1invariant} (or by observation).  
\end{proof}

\begin{proposition}
  The representation $\mb 3$ is isomorphic to the regular representation of $\Z/3$.
\end{proposition}

\begin{proof}
  As discussed, $\mb 3$ has basis $\{e_0,e_1,e_2\}$, and the generator $g$ of $\mb 3$ sends $e_0$ to $e_1$ to $e_2$ back to $e_0$. Meanwhile, by Definition~\ref{def:regular}, the regular representation $\R[\Z/3]$ has basis $\Z/3=\{e,g,g^2\}$, and $g$ sends $e$ to $g$ to $g^2$ back to $e$. Therefore, the linear map sending $e_0,e_1,e_2$ to $e,g,g^2$, respectively, is a $\Z/3$-equivariant isomorphism between $\mb 3$ and $\R[\Z/3]$.
\end{proof}

\begin{corollary}\label{cor:z3irreducible}
  The only irreducible representations of $\Z/3$ are $\mb 1$ and $\mb 2$.
\end{corollary}

\begin{proof}
  By Proposition~\ref{prop:irreducibleregular}, every irreducible representation of $\Z/3$ is a subrepresentation of $\R[\Z/3]\cong\mb 3\cong\mb 1\oplus\mb 2$.
\end{proof}

Consequently, by Maschke's theorem, we have the following corollary.

\begin{corollary}\label{cor:z3decomposition}
  Any finite-dimensional representation of $\Z/3$ is isomorphic to $m\mb1\oplus n\mb 2$ for some nonnegative integers $m$ and $n$.
\end{corollary}

Applying Corollaries~\ref{cor:inducetrivial} and \ref{cor:tensorregular} to the regular representation $\mb 3$ of $\Z/3$, we obtain the following results.

\begin{corollary}\label{cor:ind13}
  Let $e$ be the trivial subgroup of $\Z/3$ and $\mb 1$ the trivial $e$-representation. Then
  \begin{equation*}
    \Ind_e^{\Z/3}\mb1\cong\mb3.
  \end{equation*}
\end{corollary}

\begin{corollary}\label{cor:tensor3}
  If $V$ is a $\Z/3$-representation, then $V\otimes\mb3\cong(\dim V)\mb 3$.
\end{corollary}

\subsection{Finite element exterior calculus and its symmetries}\label{subsec:feec}

Arnold, Falk, and Winther define the $\cP_r\Lambda^k(\mathcal T)$ and $\cP_r^-\Lambda^k(\mathcal T)$ of $k$-forms with piecewise polynomial coefficients of degree at most $r$ on a simplicial triangulation $\mathcal T$. For a more detailed introduction to these spaces and for the correspondences between these spaces and standard finite element spaces of scalar fields and vector fields, see \cite{afw06}, as well as \cite{a13, afw09, afw10}. In this paper, we focus on a single simplex $T$, which we may map via barycentric coordinates to a standard reference simplex $T^n\subset\R^{n+1}$ consisting of nonnegative tuples $(\lambda_0,\dotsc,\lambda_n)$ with $\lambda_0+\dotsb+\lambda_n=1$. In this context, $\cP_r\Lambda^k(T^n)$ is the space of all $k$-forms on $T^n$ with polynomial coefficients of degree at most $r$, and $\cP_r^-\Lambda^k(T^n)$ is a slightly smaller space, between $\cP_{r-1}\Lambda^k(T^n)$ and $\cP_r\Lambda^k(T^n)$. For the original definition, see \cite{afw06}; however, in this context we give a natural definition from \cite{bk21feec}.

\begin{definition}
  Let $\cP_r\Lambda^k(\R^{n+1})$ denote $k$-forms on $\R^{n+1}$ with polynomial coefficients of degree at most $r$. If $r<0$, we let $\cP_r\Lambda^k(\R^{n+1})=0$. Let $X$ denote the position vector field on $\R^{n+1}$; that is, the value of $X$ at any point $(\lambda_0,\dotsc,\lambda_n)\in\R^{n+1}$ is just the vector with components $\langle\lambda_0,\dotsc,\lambda_n\rangle$. For a $(k+1)$-form $\alpha$, let $i_X\alpha$ denote the contraction with $X$, that is, $i_X\alpha$ is the $k$-form defined by $i_X\alpha(Y_1,\dotsc,Y_k)=\alpha(X, Y_1, \dotsc, Y_k)$. Let $\cP_r^-\Lambda^k(\R^{n+1})$ be the image of $\cP_{r-1}\Lambda^{k+1}(\R^{n+1})$ under $i_X$.

  As before, let $T^n$ be the subset of $\R^{n+1}$ consisting of those nonnegative tuples $(\lambda_0,\dotsc,\lambda_n)$ such that $\lambda_0+\dotsb+\lambda_n=1$. Given a $k$-form on $\R^{n+1}$, we can restrict it to a $k$-form on $T^n$. Let $\cP_r\Lambda^k(T^n)$ and $\cP_r^-\Lambda^k(T^n)$ denote the images of $\cP_r\Lambda^k(\R^{n+1})$ and $\cP_r^-\Lambda^k(\R^{n+1})$ under this restriction map, respectively.
\end{definition}

The symmetric group $S_{n+1}$ acts naturally on $\R^{n+1}$ by permuting the coordinates. This action induces an isometry of $T^n$ that permutes the vertices. A map $T^n\to T^n$ induces a pullback map $\Lambda^k(T^n)\to\Lambda^k(T^n)$. One can check using the invariance of the vector field $X$ that both the $\cP_r\Lambda^k(T^n)$ and $\cP_r^-\Lambda^k(T^n)$ spaces are preserved under this pullback map. However, because pullback is contravariant, we must be somewhat careful in how we define the action of $S_{n+1}$ on $\Lambda^k(T^n)$ so that it has the structure of an $S_{n+1}$-representation.

\begin{definition}
  Let $\pi$ be a permutation in $S_{n+1}$, so, with notation as before, $\pi$ is a bijection from the set $\{0,\dotsc,n\}$ to itself. We define a corresponding map $T^n\to T^n$ via
  \begin{equation*}
    S_\pi(\lambda_0,\dotsc,\lambda_n)=(\lambda_{\pi(0)},\dotsc,\lambda_{\pi(n)}).
  \end{equation*}
\end{definition}

\begin{remark}
  Let $v_i$ be the $i$th vertex of $T^n$, that is, $v_i=(0,\dotsc,0,1,0,\dotsc,0)$, where the $1$ is in the $i$th position. Observe that the image of $v_i$ under $S_\pi$ depends on the \emph{inverse} of $\pi$ at $i$, that is, $S_\pi(v_i)=v_{\pi^{-1}(i)}$.
\end{remark}

\begin{proposition}[\cite{li19}]
  If $\pi$ and $\sigma$ are two permutations in $S_{n+1}$, then we have the contravariant formula $S_{\pi\circ\sigma}=S_\sigma\circ S_\pi$. Consequently, the pullback maps $\Lambda^k(T^n)\to\Lambda^k(T^n)$ satisfy the covariant formula $S_{\pi\circ\sigma}^*=S_\pi^*\circ S_\sigma^*$. Thus, $\Lambda^k(T^n)$ is an $S_{n+1}$-representation with respect to this action, and hence so are the spaces $\cP_r\Lambda^k(T^n)$ and $\cP_r^-\Lambda^k(T^n)$.
\end{proposition}

\begin{remark}
  Concretely, one can check that $S_\pi^*(\lambda_i)=\lambda_{\pi(i)}$. Since pullbacks commute with the exterior derivative, we have $S_\pi^*(d\lambda_i)=d\lambda_{\pi(i)}$. Using the fact that pullbacks commute with wedge products, we can see that, more generally, $S_\pi^*$ acts on $\cP_r\Lambda^k(T^n)$ by replacing all instances of $\lambda_i$ (including in $d\lambda_i$) with $\lambda_{\pi(i)}$.
\end{remark}

We give some relevant examples. We begin with constant one-forms.

\begin{proposition}\label{prop:p0l1}
  The representation $\cP_0\Lambda^1(T^n)$ is isomorphic to the standard representation $\mb n$ of $S_{n+1}$.
\end{proposition}

\begin{proof}
  Using the notation in Definition~\ref{def:permuterepresent}, consider the map $\mb{(n+1)}\to\cP_0\Lambda^1(T^n)$ defined by $e_i\mapsto d\lambda_i$. This map is equivariant, as $\pi e_i=e_{\pi(i)}$ and $S_\pi^*d\lambda_i=d\lambda_{\pi(i)}$. This map is surjective, and its kernel is the space spanned by $e_0+\dotsb+e_n$, which maps to $d\lambda_0+\dotsb+d\lambda_n=d(\lambda_0+\dotsb+\lambda_n)=d(1)=0$. Therefore, this map is an isomorphism when restricted to the orthogonal complement of the kernel, namely $\mb n\to\cP_0\Lambda^1(T^n)$.
\end{proof}

Recalling Definition~\ref{def:1p} and Example~\ref{eg:tensor1p}, our next example is constant $(n-1)$-forms.

\begin{proposition}\label{prop:p0ln-1}
  The representation $\cP_0\Lambda^{n-1}(T^n)$ is isomorphic to $\mb n\otimes\mb1'$.
\end{proposition}

\begin{proof}
  As vector spaces, the Hodge star on $T^n$ is an isomorphism between $\cP_0\Lambda^{n-1}(T^n)$ and $\cP_0\Lambda^1(T^n)$. If $\pi$ is an even permutation, then $S_\pi$ is a rotation, and if $\pi$ is an odd permutation, then $S_\pi$ is a reflection. The Hodge star commutes with rotations and anti-commutes with reflections. Thus, $S_\pi^*(*\alpha)=*(S_\pi^*\alpha)$ if $\pi$ is even and $S_\pi^*(*\alpha)=-{*(S_\pi^*\alpha)}$ if $\pi$ is odd. Thus, as $S_{n+1}$-representations, the Hodge star gives an isomorphism $\cP_0\Lambda^{n-1}(T^n)\cong\cP_0\Lambda^1(T^n)\otimes\mb 1'\cong\mb n\otimes\mb1'$.
\end{proof}

\begin{remark}\label{rem:cross}
  We can, more generally, view $\mb n$ as the standard representation of the orthogonal group $O(n)$, of which $S_{n+1}$ is a subgroup. In this context, vectors in $\mb n\otimes\mb 1'$ are called \emph{pseudovectors}: they rotate like normal vectors, but upon reflection they acquire an additional minus sign. A standard example is cross-products. If we rotate $v$ and $w$ in $\R^3$ to $v'$ and $w'$, then if we apply that rotation to $v\times w$ we arrive at $v'\times w'$. In contrast, if we reflect $v$ and $w$ to $v'$ and $w'$, then if we apply the same reflection to $v\times w$ we arrive at $-v'\times w'$.
\end{remark}

Next, we discuss the concept of trace in this context.

\begin{definition}
  Let $F$ be a face of a simplex $T$. Then the inclusion map $F\hookrightarrow T$ induces a restriction or \emph{trace} map $\Lambda^k(T)\to\Lambda^k(F)$, which we denote $\tr_{T, F}$.
\end{definition}

\begin{definition}
  Let $\oP_r\Lambda^k(T^n)$ and $\oP_r^-\Lambda^k(T^n)$ denote those forms in $\cP_r\Lambda^k(T^n)$ and $\cP_r^-\Lambda^k(T^n)$, respectively, that have vanishing trace on $\partial T^n$, that is, vanishing trace on the $(n-1)$-dimensional faces of $T^n$.
\end{definition}

One can check that because permutations in $S_{n+1}$ permute the faces of $T^n$, the spaces $\oP_r\Lambda^k(T^n)$ and $\oP_r^-\Lambda^k(T^n)$ are preserved under the action of $S_{n+1}$ and are thus $S_{n+1}$-representations in their own right.

\subsubsection*{The duality isomorphism} We are now ready to discuss the duality isomorphism that is one of the key ingredients in finite element exterior calculus. Without considering the symmetric group, this isomorphism was first discussed in \cite{afw06} and developed in greater depth in \cite{bk21feec, li18}. Viewing these spaces as $S_{n+1}$-representations, this isomorphism was first proved by Licht \cite{li19}, and we provide an alternate proof here. It may be helpful to recall Definition~\ref{def:1p} and Example~\ref{eg:tensor1p}.

\begin{proposition}[{reformulation of \cite[Theorem 2]{li19}}]\label{prop:popiso}
  We have the isomorphisms of $S_{n+1}$-representations
  \begin{align*}
    \cP_r\Lambda^k(T^n)&\cong\oP_{r+k+1}^-\Lambda^{n-k}(T^n)\otimes\mb1',&\cP_r^-\Lambda^k(T^n)&\cong\oP_{r+k}\Lambda^{n-k}(T^n)\otimes\mb1',
  \end{align*}
  except for the case $\cP_0^-\Lambda^0(T^n)\cong0\not\cong\oP_0\Lambda^n(T^n)$.
\end{proposition}

\begin{proof}[Alternate proof]
  In \cite[Corollary 3.3]{bk21feec}, we construct vector space isomorphisms $\cP_r\Lambda^k(T^n)\to\oP_{r+k+1}^-\Lambda^{n-k}(T^n)$ and $\cP_r^-\Lambda^k(T^n)\to\oP_{r+k}\Lambda^{n-k}(T^n)$. Specifically both of these isomorphisms are the function $(\Phi^*)^{-1}\circ(u_0\dotsm u_n){*_{S^n}}\circ\Phi^*$, where the notation is as follows. The map $\Phi$ maps the unit sphere $S^n\subset\R^{n+1}$ to the standard simplex $T^n\subset\R^{n+1}$ by sending $(u_0,\dotsc,u_n)\mapsto(u_0^2,\dotsc,u_n^2)$. We can thus think of the pullback map $\Phi^*\colon\Lambda^k(T^n)\to\Lambda^k(S^n)$ as the change of coordinates defined by $\lambda_i=u_i^2$, $d\lambda_i=2u_i\,du_i$. The map $*_{S^n}\colon\Lambda^k(S^n)\to\Lambda^{n-k}(S^n)$ is the Hodge star on the unit sphere, and $u_0\dotsm u_n$ denotes multiplication by the product of all the coordinate functions.

  It remains then to see what happens to the map $(\Phi^*)^{-1}\circ(u_0\dotsm u_n){*_{S^n}}\circ\Phi^*$ when we permute the coordinate directions. The map $\Phi$ commutes with permuting the coordinate directions. Likewise, the bubble function $u_0\dotsm u_n$ is invariant under permuting the coordinate directions. Finally, the Hodge star $*_{S^n}$ commutes with \emph{orientation-preserving} isometries of $S^n$, but \emph{anti-commutes} with orientation-reversing isometries. Whether or not a permutation of the coordinates preserves or reverses orientation is precisely the sign of that permutation. We conclude that the maps $\cP_r\Lambda^k(T^n)\to\oP_{r+k+1}^-\Lambda^{n-k}(T^n)$ and $\cP_r^-\Lambda^k(T^n)\to\oP_{r+k}^-\Lambda^{n-k}(T^n)$ commute with even permutations of the coordinates and anti-commute with odd permutations of the coordinates, as desired.
\end{proof}

\section{Geometric Decompositions}\label{sec:geodecomp}
The geometric decomposition of Arnold, Falk, and Winther \cite{afw06, afw09} decomposes the $\cP_r\Lambda^k(T^n)$ and $\cP_r^-\Lambda^k(T^n)$ spaces as a direct sum of spaces of forms on the faces $F$ of the simplex $T^n$. This result is key to defining these finite element spaces on a triangulation, and also key to Licht's recursive basis construction \cite{li19}. In this section, we distill the essential features of the geometric decomposition, giving a general theory that is well-suited to understanding the representation theory of these spaces, thereby helping us not only to construct invariant bases but also to prove when such bases do not exist.

We discuss the geometric decomposition in two subsections: first we discuss the vector space decomposition, and then we show that it is in fact a decomposition of $S_{n+1}$-representations. Subsection~\ref{subsec:linearalg} is essentially a generalization of \cite[Section 4]{afw09}, and Subsection~\ref{subsec:representations} is essentially a generalization of \cite[Section 7]{li19}. We establish our theory for a general finite-dimensional subspace $V\subset\Lambda^k(T^n)$, but, in practice, $V$ will be either $\cP_r\Lambda^k(T^n)$ or $\cP_r^-\Lambda^k(T^n)$.

\subsection{Vector space decomposition}\label{subsec:linearalg}
\begin{definition}
  Given a finite-dimensional subspace $V\subset\Lambda^k(T^n)$, for a face $F\subseteq T^n$, let $V(F)=\tr_{T^n, F}(V)$, and let $\oV(F)$ be those forms in $V(F)$ whose traces vanish on $\partial F$.
\end{definition}

\begin{example}
  If $V=\cP_r\Lambda^k(T^n)$, then $V(F)=\cP_r\Lambda^k(F)$, and $\oV(F)=\oP_r\Lambda^k(F)$, and likewise for the $\cP_r^-\Lambda^k(T^n)$ spaces. (For the compatibility of $\cP_r^-\Lambda^k$ with traces, see \cite[Subsection 3.2]{afw09}.)
\end{example}

\begin{definition}\label{def:Vd}
  We define a filtration
  \begin{equation*}
    V = V_0 \supseteq V_1\supseteq\dotsb\supseteq V_n=\oV(T^n)\supseteq V_{n+1}=0
  \end{equation*}
  by letting $V_d$ denote those forms in $V$ whose traces vanish on all of the $(d-1)$-dimensional faces of $T^n$.
\end{definition}

\begin{example}\label{eg:P2}
  If $V=\cP_2\Lambda^0(T^2)$, then
  \begin{equation*}
    V=V_0\supseteq V_1=\Span\{\lambda_1\lambda_2,\lambda_2\lambda_0,\lambda_0\lambda_1\}\supseteq V_2=0\supseteq V_3=0.
  \end{equation*}
\end{example}

\begin{definition}\label{def:W}
  We let $W_d$ be a complement to $V_{d+1}$ in $V_d$, so $V_d=W_d\oplus V_{d+1}$. For now, the particular choice of $W_d$ does not matter, but we note that there are some natural options.
  \begin{itemize}
  \item One natural option is to use orthogonal complements with respect to the the inner product $(\alpha,\beta)\mapsto\int_{T^n}\langle\alpha,\beta\rangle$, that is, we define $W_d$ to contain those forms in $V_d$ that are orthogonal to $V_{d+1}$ with respect to this inner product.
  \item Another natural option, if available, is to use consistent extension operators as defined in \cite[Section 4]{afw09}. We discuss this option in more detail in Appendix~\ref{sec:extension}.
  \end{itemize}
\end{definition}

\begin{proposition}\label{prop:Wdecomp}
\begin{equation*}
  V = W_0\oplus\dotsb\oplus W_n.
\end{equation*}
\end{proposition}

\begin{proof}
  The proposition follows from induction on the claim that $V_{n-d}=W_{n-d}\oplus\dotsb\oplus W_n$.
\end{proof}

To further decompose the $W_d$, we need the following lemma.

\begin{lemma}\label{lemma:traceW}
  If $\alpha\in W_d$ and $F$ is a $d$-dimensional face of $T^n$, then $\tr_{T^n, F}\alpha$ has vanishing trace on $\partial F$.
\end{lemma}
\begin{proof}
  The boundary of $F$ is composed of $(d-1)$-dimensional faces of $T^n$. Since $W_d\subseteq V_d$, forms in $W_d$ have vanishing trace on $(d-1)$-dimensional faces of $T^n$ by definition.
\end{proof}

Lemma~\ref{lemma:traceW} yields a map $W_d\to\oV(F)$. Taking the direct sum of these maps for all $F$, we obtain our geometric decomposition.

\begin{definition}\label{def:geodecomp}
  Given complements $W_d$ as in Definition~\ref{def:W}, the \emph{geometric decomposition map}
  \begin{equation*}
    \mathcal D\colon V\to\bigoplus_{d=0}^n\bigoplus_{\substack{F\subseteq T^n\\\dim F = d}}\oV(F)
  \end{equation*}
  is defined as follows.

  For $\alpha\in W_d$, let
  \begin{equation}\label{eq:geodecomp}
    \mathcal D(\alpha)=\bigoplus_{\substack{F\subseteq T^n\\\dim F = d}}\tr_{T^n, F}\alpha\in\bigoplus_{\substack{F\subseteq T^n\\\dim F = d}}\oV(F),
  \end{equation}
  and extend by linearity to $V$.
\end{definition}

\begin{remark}
  Note that our definition of the geometric decomposition map depends not only on $V$ but also on the choice of complements $W_d$. This choice is not unique, and different choices will yield different decomposition maps.
\end{remark}

Nonetheless, the useful properties of the map $\mathcal D$ do not depend on these choices.

\begin{proposition}\label{prop:injective}
  The geometric decomposition map $\mathcal D$ is injective.
\end{proposition}

\begin{proof}
  It suffices to prove that if $\alpha\in W_d$ and $\tr_{T^n, F}\alpha=0$ for all $d$-dimensional faces $F$, then $\alpha=0$. This trace condition implies that $\alpha\in V_{d+1}$ by definition. Since $\alpha$ is in both summands of $V_d=W_d\oplus V_{d+1}$, it must be zero.
\end{proof}

\begin{remark}
  We usually want the geometric decomposition map to be an isomorphism, but proving surjectivity requires assumptions. Indeed, one can check that the map $\mathcal D$ is not an isomorphism when $V=\cP_0\Lambda^k(T^n)$ and $0\le k<n$.
\end{remark}

\begin{proposition}\label{prop:Diso}
  Let $r\ge1$. If $V$ is $\cP_r\Lambda^k(T^n)$ or $\cP_r^-\Lambda^k(T^n)$, then the geometric decomposition map $\mathcal D$ is a vector space isomorphism.
\end{proposition}
\begin{proof}
  In light of Proposition~\ref{prop:injective}, we only need to know that the dimensions of the domain and codomain of $\mathcal D$ are equal. This fact follows by combining \cite[Theorem 4.9, Theorem 4.13, Theorem 4.16, Theorem 4.22]{afw06}.
\end{proof}

\begin{remark}\label{rem:extension}
  In the case where the geometric decomposition map
  \begin{equation*}
    \mathcal D\colon V\to\bigoplus_{d=0}^n\bigoplus_{\substack{F\subseteq T^n\\\dim F = d}}\oV(F)=\bigoplus_{F\subseteq T^n}\oV(F)
  \end{equation*}
  is surjective (and thus an isomorphism), then for each face $F\subseteq T^n$ we obtain inverse maps
  \begin{equation*}
    E_{F,T^n}\colon\oV(F)\to V.
  \end{equation*}
  As suggested by the notation, these are extension operators, in the sense that $\tr_{T^n,F}\circ E_{F, T^n}$ is the identity, as we will soon see.
\end{remark}

\begin{proposition}
  Let $\alpha\in V$, and assume that the geometric decomposition map $\mathcal D$ sends $\alpha\in V$ to a single summand, that is, $\mathcal D(\alpha)=\beta$ for some $\beta\in\oV(F)$ for some $d$-dimensional face $F$ of $T^n$. Then the following holds.
  \begin{itemize}
  \item $\tr_{T^n, F}\alpha=\beta$.
  \item $\tr_{T^n,F'}=0$ for all faces $F'$ such that $\dim F'<d$.
  \item $\tr_{T^n,F'}=0$ for all faces $F'\neq F$ such that $\dim F'=d$.
  \end{itemize}
\end{proposition}

\begin{proof}
  Since the geometric decomposition map is defined as a direct sum of maps and is injective, if $\mathcal D(\alpha)$ is in a single summand then $\alpha$ itself must be in a single summand, namely $W_d$. Since $W_d\subseteq V_d$, the trace of $\alpha$ vanishes on all $(d-1)$-dimensional faces of $T^n$, and hence also on all faces of $T^n$ with even smaller dimension.

  Next, considering Equation~\eqref{eq:geodecomp} in Definition~\ref{def:geodecomp}, since $\mathcal D(\alpha)$ lands in a single summand, we have that $\tr_{T^n,F'}\alpha=0$ for all $d$-dimensional faces $F'$ of $T^n$ other than $F$.

  Finally, considering the summand that $\mathcal D(\alpha)$ does land in, we have that $\beta=\mathcal D(\alpha)=\tr_{T^n,F}\alpha$ as desired.
\end{proof}

\begin{corollary}\label{cor:extensionproperties}
  If the geometric decomposition map is an isomorphism and $E_{F,T^n}$ is defined as in Remark~\ref{rem:extension}, then if $\beta\in\oV(F)$, we have the following.
  \begin{itemize}
  \item $\tr_{T^n, F}E_{F, T^n}\beta=\beta$.
  \item $\tr_{T^n,F'}E_{F, T^n}\beta=0$ for all faces $F'$ such that $\dim F'<\dim F$.
  \item $\tr_{T^n,F'}E_{F, T^n}\beta=0$ for all faces $F'\neq F$ such that $\dim F'=\dim F$.
  \end{itemize}
\end{corollary}

\begin{proof}
  Let $\alpha=E_{F, T^n}\beta$. By definition of $E_{F, T^n}$ as an inverse map, we have $\mathcal D(\alpha)=\beta$, which lies in a single summand of $\bigoplus\oV(F)$.
\end{proof}

In particular, by Proposition~\ref{prop:Diso}, we have extension operators for the standard finite element spaces that enjoy all of the properties in Corollary~\ref{cor:extensionproperties}. However, we remark that the notion of a \emph{consistent family} of extension operators from \cite{afw09} is strictly stronger than what we have here. We discuss the relationship between these extension operators in Appendix~\ref{sec:extension}. Specifically, if we start with a consistent family of extension operators, then Proposition~\ref{prop:extensiontoW} shows how to choose the $W_d$ in Definition~\ref{def:W} to recover that notion of geometric decomposition here. On the other hand, Remark~\ref{rem:converse} highlights the fact that, for a consistent family of extension operators, $\tr_{T^n,F'}E_{F,T^n}\beta=0$ for all faces $F'$ not containing $F$, which is a stronger property than Corollary~\ref{cor:extensionproperties}. The remark gives a specific example showing that, at the cost of having weaker properties, the geometric decomposition theory in this paper is strictly more general than the geometric decomposition theory stemming from consistent families of extension operators.

\subsection{The action of $S_{n+1}$}\label{subsec:representations}

We now assume that $V$ is $S_{n+1}$-invariant and show that the geometric decomposition map is equivariant, assuming that the choice of the complements $W_d$ respects the symmetry.

First, observe that the spaces $V_d$ are also $S_{n+1}$-invariant. Indeed, the action of any $\pi\in S_{n+1}$ permutes the $(d-1)$-dimensional faces of $T^n$, so if $\alpha$ has vanishing trace on all of them, then so does $S_\pi^*\alpha$.

Next, we must choose the complements $W_d$ to be $S_{n+1}$-invariant. Finding invariant complements is always possible by Maschke's theorem, but we can also use the explicit choices given in Definition~\ref{def:W}. Indeed, observe that the inner product $(\alpha,\beta)\mapsto\int_{T^n}\langle\alpha,\beta\rangle$ is invariant under isometries, so if we set $W_d$ to be the orthogonal complement of $V_{d+1}$ in $V_d$, then $W_d$ is $S_{n+1}$-invariant. For the construction via extension operators, see Appendix~\ref{sec:extension}.

To show that the geometric decomposition map $\mathcal D$ is equivariant, we must first understand the natural action of the permutation group $S_{n+1}$ on the space $\bigoplus\oV(F)$. Our first observation is that for $\pi\in S_{n+1}$, the map $S_\pi$ sends a face $F$ to a face $F'$ (possibly the same as $F$) of the same dimension. Moreover, $S_\pi$ sends $\partial F$ to $\partial F'$. We thus have the following commutative diagram, yielding a corresponding diagram of pullback maps.
\begin{equation*}
  \begin{tikzcd}
    T^n&\arrow[l, "S_\pi", swap]T^n&&\Lambda^k(T^n)\arrow[r, "S_\pi^*"]\arrow[d,"\tr_{T^n, F'}"]&\Lambda^k(T^n)\arrow[d,"\tr_{T^n, F}"]\\
    F'\arrow[u, hook]&F\arrow[l, "S_\pi\rvert_F", swap]\arrow[u, hook]&&\Lambda^k(F')\arrow[r, "S_\pi^*"]\arrow[d,"\tr_{F',\partial F'}"]&\Lambda^k(F)\arrow[d,"\tr_{F,\partial F}"]\\
    \partial F'\arrow[u, hook]&\partial F\arrow[l, "S_\pi\rvert_{\partial F}", swap]\arrow[u, hook]&&\Lambda^k(\partial F')\arrow[r, "S_\pi^*"]&\Lambda^k(\partial F)
  \end{tikzcd}
\end{equation*}
Here we abuse notation and abbreviate $(S_\pi\rvert_F)^*$ and $(S_\pi\rvert_{\partial F})^*$ to just $S_\pi^*$. Next, because $V$ is $S_{n+1}$-invariant and because $V(F)$ is defined to be the image of $V$ under the trace map, we can restrict the above diagram to the following.
\begin{equation}\label{eq:Vdiagram}
  \begin{tikzcd}
    V\arrow[r, "S_\pi^*"]\arrow[d,"\tr_{T^n, F'}"]&V\arrow[d,"\tr_{T^n, F}"]\\
    V(F')\arrow[r, "S_\pi^*"]\arrow[d,"\tr_{F',\partial F'}"]&V(F)\arrow[d,"\tr_{F,\partial F}"]\\
    \Lambda^k(\partial F')\arrow[r, "S_\pi^*"]&\Lambda^k(\partial F)
  \end{tikzcd}
\end{equation}
Moreover, the lower part of the diagram shows that $S_\pi^*$ preserves the vanishing trace property; that is, if $\alpha\in\oV(F')$, then $S_\pi^*\alpha\in\oV(F)$. In short, we have a map $S_\pi^*\colon\oV(F')\to\oV(F)$, which induces an action
\begin{equation*}
  S_\pi^*\colon\bigoplus_{\substack{F\subseteq T^n\\\dim F = d}}\oV(F)\to\bigoplus_{\substack{F\subseteq T^n\\\dim F = d}}\oV(F).
\end{equation*}
It is at this point that we remark that $S_{n+1}$ acts transitively on the $d$-dimensional faces of $T^n$ and that the stabilizer of $T^d\subseteq T^n$ is $S_{d+1}\subseteq S_{n+1}$. Consequently, with respect to the above action, $\oV(T^d)$ is an $S_{d+1}$-invariant subspace, and each $\oV(F)$ is a translate of $\oV(T^d)$. We conclude that
\begin{equation*}
  \bigoplus_{\substack{F\subseteq T^n\\\dim F=d}}\oV(F)=\Ind_{S_{d+1}}^{S_{n+1}}\oV(T^d).
\end{equation*}

We can now prove our desired claim.
\begin{proposition}\label{prop:Dequi}
  If $V$ is $S_{n+1}$-invariant and the $W_d$ are chosen to be $S_{n+1}$-invariant subspaces, then the geometric decomposition map $\mathcal D$ is equivariant.
\end{proposition}

\begin{proof}
  Since the $W_d$ are assumed to be $S_{n+1}$-invariant, it suffices to show equivariance of the map
  \begin{align*}
    W_d&\to\bigoplus_{\substack{F\subseteq T^n\\\dim F = d}}\oV(F)\\
    \alpha&\mapsto\bigoplus_{\substack{F\subseteq T^n\\\dim F = d}}\tr_{T^n, F}\alpha
  \end{align*}
  This claim follows by restricting diagram~\eqref{eq:Vdiagram} to $W_d$, obtaining
  \begin{equation*}
    \begin{tikzcd}
      W_d\arrow[r, "S_\pi^*"]\arrow[d,"\tr_{T^n, F'}"]&W_d\arrow[d,"\tr_{T^n, F}"]\\
      \oV(F')\arrow[r, "S_\pi^*"]\arrow[d,"\tr_{F',\partial F'}"]&\oV(F)\arrow[d,"\tr_{F,\partial F}"]\\
      0\arrow[r, "S_\pi^*"]&0
    \end{tikzcd}
  \end{equation*}
  Here, $\dim F=d$, and we used Lemma~\ref{lemma:traceW}, along with the above observation that $S_\pi^*$ sends $\oV(F')$ to $\oV(F)$.
\end{proof}

\begin{corollary}\label{cor:geodecomp}
  For $r\ge1$, we have the following isomorphisms of $S_{n+1}$-representations.
  \begin{align*}
    \cP_r\Lambda^k(T^n)&\cong\bigoplus_{d=0}^n\bigoplus_{\substack{F\subseteq T^n\\\dim F = d}}\oP_r\Lambda^k(F)\cong\bigoplus_{d=0}^n\Ind_{S_{d+1}}^{S_{n+1}}\oP_r\Lambda^k(T^d),\\
    \cP_r^-\Lambda^k(T^n)&\cong\bigoplus_{d=0}^n\bigoplus_{\substack{F\subseteq T^n\\\dim F = d}}\oP_r^-\Lambda^k(F)\cong\bigoplus_{d=0}^n\Ind_{S_{d+1}}^{S_{n+1}}\oP_r^-\Lambda^k(T^d).
  \end{align*}
\end{corollary}

\begin{proof}
  Proposition~\ref{prop:Diso} tells us that $\mathcal D$ is a vector space isomorphism and Proposition~\ref{prop:Dequi} tells us that $\mathcal D$ is equivariant.
\end{proof}

As a result, we obtain the following claim, which is also discussed by Licht in \cite[Section 8]{li19}.

\begin{corollary}\label{cor:induct}
  Let $r\ge1$. If $\oP_r\Lambda^k(T^d)$ has an $S_{d+1}$-invariant basis up to sign for every $d\le n$, then $\cP_r\Lambda^k(T^n)$ has an $S_{n+1}$-invariant basis up to sign. Likewise, the same holds for the $\cP^-$ spaces.
\end{corollary}

\begin{proof}
  Proposition~\ref{prop:induce} tells us that the property of having an invariant basis up to sign is preserved under induction, and Proposition~\ref{prop:sumtensor} tells us that this property is preserved under direct sums.
\end{proof}

As Licht discusses at the end of \cite{li19}, the completeness of his invariant basis construction is equivalent to the converse of Corollary~\ref{cor:induct}. As we will see, the converse holds in dimension two, but not in dimension three.

\section{Invariant bases}\label{sec:results}
Here, we prove which finite element exterior calculus spaces have invariant bases and which do not, first in dimension two and then in dimension three. At the end, we give some remarks about explicitly constructing these bases when it is possible to do so, as well as how one could approach proving analogous results in higher dimensions.

\subsection{Base cases}\label{subsec:invariant}
As illustrated by our example in the introduction and noted in \cite{li19}, the case of scalar functions is easy.
\begin{proposition}\label{prop:k0n}
  The spaces $\cP_r\Lambda^0(T^n)=\cP_r^-\Lambda^0(T^n)$ and the spaces $\cP_r\Lambda^n(T^n)=\cP_{r+1}^-\Lambda^n(T^n)$ always have $S_{n+1}$-invariant bases up to sign.
\end{proposition}

\begin{proof}
  One can check that the monomials in the variables $\lambda_0,\dotsc,\lambda_n$ of degree exactly $r$ form a basis for $\cP_r\Lambda^0(T^n)$; this basis is $S_{n+1}$-invariant. Multiplying these monomials by the volume form on $T^n$, we obtain a basis for $\cP_r\Lambda^n(T^n)$; this basis is $S_{n+1}$-invariant up to sign.
\end{proof}

Hence, our focus will be on vector fields, which in dimension two correspond to one-forms and in dimension three correspond to one-forms or two-forms.

In three dimensions, something very special happens: there is a basis of $\R^3$ that is invariant up to sign with respect to the group $S_4$ of symmetries of the regular tetrahedron. A geometric way of seeing this basis is as follows. The tetrahedron has three pairs of opposite edges. For each such pair, consider the vector joining the midpoints of the two edges. (This vector is determined up to sign, which can be chosen arbitrarily.) One can see that these three vectors form a basis of $\R^3$. Moreover, the symmetries of a tetrahedron send pairs of opposite edges to pairs of opposite edges, and hence preserve this basis up to sign. An alternative geometric viewpoint is that, starting from a cube, we can select four vertices so that they are pairwise non-adjacent. These vertices determine a regular tetrahedron, so the group of symmetries of the regular tetrahedron is a subgroup of the group of symmetries of a cube. From this perspective, the edge directions of the cube are the invariant basis up to sign.

Recalling Example~\ref{eg:equilateral}, the above discussion shows that $\mb 3$ has an $S_4$-invariant basis up to sign, and thus so does $\mb 3\otimes\mb1'$ by Proposition~\ref{prop:tensorsign}. Recalling Propositions~\ref{prop:p0l1} and \ref{prop:p0ln-1}, we obtain the following result, which also appears in \cite{li19} with algebraic justification.

\begin{proposition}[{\cite[Lemmas 9 and 10]{li19}}]\label{prop:constdim3}
  The spaces $\cP_0\Lambda^1(T^3)$ and $\cP_0\Lambda^2(T^3)$ have $S_4$-invariant bases up to sign.
\end{proposition}

This proposition has far-reaching consequences.

\begin{corollary}\label{cor:3dbasis}
  The spaces $\cP_r\Lambda^k(T^3)$ have $S_4$-invariant bases up to sign for all $r$ and $k$.
\end{corollary}

\begin{proof}
  Observe that $\cP_r\Lambda^k(T^3)\cong\cP_r\Lambda^0(T^3)\otimes\cP_0\Lambda^k(T^3)$ and apply Propositions~\ref{prop:sumtensor}, \ref{prop:k0n}, and \ref{prop:constdim3}.
\end{proof}

\subsection{The obstruction}\label{subsec:Z3}
Our main results include not only constructions of invariant bases but also obstructions to the existence of invariant bases. Fundamentally, our obstruction stems from the following easy fact: There is no basis for $\R^2$ that is invariant up to sign with respect to the group of symmetries $S_3$ of the equilateral triangle. Indeed, it suffices to consider the subgroup $\Z/3$ of $S_3$; there is no way to pick two vectors in $\R^2$ so that $120^\circ$ rotations take each of them to the other up to sign. Note the contrast between the situation here and the exceptional situation in dimension three discussed in Subsection~\ref{subsec:invariant}.

Recalling Corollary~\ref{cor:z3decomposition}, any finite-dimensional representation of $\Z/3$ is isomorphic to $m\mb1\oplus n\mb2$ for some nonnegative integers $m$ and $n$. The main result of this subsection is the following characterization of which $\Z/3$-representations have invariant bases up to sign.

\begin{proposition}\label{prop:obstruction}
  Let $V\cong m\mb1\oplus n\mb2$ be a $\Z/3$-representation. Then $V$ has a basis that is $\Z/3$-invariant up to sign if and only if $m\ge n$.
\end{proposition}
\begin{proof}
  If $m\ge n$, then
  \begin{equation*}
    V\cong(m-n)\mb1\oplus n\mb1\oplus n\mb2\cong(m-n)\mb1\oplus n(\mb1\oplus\mb2)\cong(m-n)\mb1\oplus n\mb3
  \end{equation*}
  Recalling Proposition~\ref{prop:13invariant}, $\mb1$ and $\mb3$ have $\Z/3$-invariant bases, so $V$ has an invariant basis by Proposition~\ref{prop:sumtensor}.

  Conversely, assume that $V$ has a basis that is $\Z/3$-invariant up to sign. Then Proposition~\ref{prop:invariantinduced} tells us that $V=\bigoplus_i\Ind_{H_i}^{\Z/3}L_i$, where $H_i$ is a subgroup of $\Z/3$ and $L_i$ is a one-dimensional representation of $H_i$. The group $\Z/3$ has only two subgroups, the trivial group and itself.

  Consider the case where $H_i$ is the trivial group $e$. The only one-dimensional representation of the trivial group is $L_i=\mb1$. Thus, in this case, $\Ind_{H_i}^{\Z/3}L_i\cong\mb3$ by Corollary~\ref{cor:ind13}. Consider now the case where $H_i=\Z/3$. Once again, the only one-dimensional representation of $\Z/3$ is the trivial representation $L_i=\mb1$. In this case, $\Ind_{H_i}^{\Z/3}L_i=\Ind_{\Z/3}^{\Z/3}\mb1=\mb1$.

  Therefore, $V\cong k\mb1\oplus n\mb3$ for some nonnegative integers $k$ and $n$. We compute
  \begin{equation*}
    V\cong k\mb1\oplus n(\mb1\oplus\mb2)\cong(k+n)\mb1\oplus n\mb2,
  \end{equation*}
  and we observe that $k+n\ge n$ as desired.
\end{proof}

\begin{corollary}\label{cor:obstruction}
  For any $n$, a representation $V$ has a $\Z/3$-invariant basis up to sign if and only if $V\oplus n\mb3$ does.
\end{corollary}

Finally, we will make frequent use of the following easy result.

\begin{proposition}\label{prop:3s}
  Let $V$ be a $\Z/3$-representation, and let $g$ be a generator of $\Z/3$. Assume that we have a vector space decomposition $V=V_0\oplus V_1\oplus V_2$ such that $g(V_0)=V_1$, $g(V_1)=V_2$, and $g(V_2)=V_0$. Then, as $\Z/3$-representations, we have $V\cong(\dim V_0)\mb3$.
\end{proposition}

\begin{proof}
  Let $\xi_1,\dotsc,\xi_n$ be a basis of $V_0$. Then $g\xi_1,\dotsc,g\xi_n$ is a basis of $V_1$ and $g^2\xi_1,\dotsc,g^2\xi_n$ is a basis of $V_2$. Thus,
  \begin{equation*}
    V=\bigoplus_{i=1}^n\Span\{\xi_i,g\xi_i,g^2\xi_i\}\cong\bigoplus_{i=1}^n\mb3,
  \end{equation*}
  as desired.

  Alternatively, let $e$ denote the trivial group, whose one-dimensional representation we also denote by $\mb1$. Thinking of $V_0$ as an $e$-representation, we see that
  \begin{equation*}
    V=\Ind_e^{\Z/3}V_0\cong\Ind_e^{\Z/3}\bigoplus_{i=1}^{\dim V_0}\mb1\cong\bigoplus_{i=1}^{\dim V_0}\Ind_e^{\Z/3}\mb1\cong\bigoplus_{i=1}^{\dim V_0}\mb3.\qedhere
  \end{equation*}
\end{proof}

\subsection{Main results}

\begin{theorem}\label{thm:2d}
  The following spaces have $S_3$-invariant bases up to sign if and only if the corresponding condition holds.
  \begin{equation*}
    \begin{array}{ccl}
      \cP_r\Lambda^1(T^2)& \quad\text{if and only if}\quad &r\notin 3\mathbb N_0,\\
      \cP_r^-\Lambda^1(T^2)& \quad\text{if and only if}\quad &r\notin 3\mathbb N_0+2,\\
      \oP_r\Lambda^1(T^2)& \quad\text{if and only if}\quad &r\notin 3\mathbb N_0+3,\\
      \oP_r^-\Lambda^1(T^2)& \quad\text{if and only if}\quad &r\notin 3\mathbb N_0+2.
    \end{array}
  \end{equation*}
\end{theorem}
\begin{proof}
  The first direction of this theorem, that is, constructing the invariant bases, is due to Licht \cite[Theorem 7]{li19}. It remains to prove the converse, that Licht's construction is complete: There are no invariant bases in the remaining cases.

  As before, let $\Z/3$ denote the subgroup of $S_3$ consisting of cyclic permutations of $\{0,1,2\}$; these give rotations of $T^2$. We observe that a basis that is $S_3$-invariant up to sign is also $\Z/3$-invariant up to sign. Our strategy, then, is to decompose the finite element spaces into irreducible $\Z/3$-representations and then to apply the obstruction to having invariant bases given by Proposition~\ref{prop:obstruction}.

  If $r\ge1$, then applying the geometric decomposition Corollary~\ref{cor:geodecomp}, we have the isomorphism of $\Z/3$-representations
  \begin{equation*}
    \cP_r\Lambda^1(T^2)\cong\left(\oP_r\Lambda^1(F_{12})\oplus\oP_r\Lambda^1(F_{20})\oplus\oP_r\Lambda^1(F_{01})\right)\oplus\oP_r\Lambda^1(T^2),
  \end{equation*}
  where $F_{ij}$ denotes the edge joining vertices $i$ and $j$. We observe that the factor $\left(\oP_r\Lambda^1(F_{12})\oplus\oP_r\Lambda^1(F_{20})\oplus\oP_r\Lambda^1(F_{01})\right)$ satisfies the hypotheses of Proposition~\ref{prop:3s}, so
  \begin{equation*}
    \cP_r\Lambda^1(T^2)\cong\left(\dim\oP_r\Lambda^1(T^1)\right)\mb3\oplus\oP_r\Lambda^1(T^2).
  \end{equation*}
  Thus, by Corollary ~\ref{cor:obstruction}, if $r\ge1$, then $\cP_r\Lambda^1(T^2)$ has an $\Z/3$-invariant representation up to sign if and only if $\oP_r\Lambda^1(T^2)$ does. Following the same reasoning, $\cP_r^-\Lambda^1(T^2)$ has a $\Z/3$-invariant representation up to sign if and only if $\oP_r^-\Lambda^1(T^2)$ does.

  Moreover, noting that the sign representation is trivial when restricted to $\Z/3\subset S_3$, Proposition~\ref{prop:popiso} tells us that $\cP_r\Lambda^1(T^2)$ has a $\Z/3$-invariant basis up to sign if and only if $\oP_{r+2}^-\Lambda^1(T^2)$ does, and $\cP_r^-\Lambda^1(T^2)$ has a $\Z/3$-invariant basis up to sign if and only if $\oP_{r+1}\Lambda^1(T^2)$ does.

  We then proceed by induction, starting from the base case that $\cP_0\Lambda^1(T^2)\cong\mb2$ does not have a $\Z/3$-invariant basis up to sign. For the inductive step, let $r\ge0$ and assume that $\cP_r\Lambda^1(T^2)$ does not have a $\Z/3$-invariant basis up to sign. Following the above observations, we conclude that $\oP_{r+2}^-\Lambda^1(T^2)$ does not have such a basis, so neither does $\cP_{r+2}^-\Lambda^1(T^2)$, so neither does $\oP_{r+3}\Lambda^1(T^2)$, and so neither does $\cP_{r+3}\Lambda^1(T^2)$. The claim follows by induction.
\end{proof}

We now prove a similar claim in dimension three, but here our list differs substantially from Licht's, so we must prove both directions.
\begin{theorem}\label{thm:3d}
  The following spaces have $S_4$-invariant bases up to sign if and only if the corresponding condition holds.
  \begin{equation*}
    \begin{array}{ccl}
      \cP_r\Lambda^1(T^3)& \text{always},\\
      \cP_r^-\Lambda^1(T^3)& \quad\text{if and only if}\quad &r\notin 3\mathbb N_0+2,\\
      \cP_r\Lambda^2(T^3)& \text{always},\\
      \cP_r^-\Lambda^2(T^3)& \text{always},\\
      \oP_r\Lambda^1(T^3)& \text{always},\\
      \oP_r^-\Lambda^1(T^3)& \text{always},\\
      \oP_r\Lambda^2(T^3)& \text{if and only if} &r\notin 3\mathbb N_0+3,\\
      \oP_r^-\Lambda^2(T^3)& \text{always}.\\
    \end{array}
  \end{equation*}
\end{theorem}

\begin{proof}[Proof, part 1: Constructing invariant bases]
  This proof essentially follows Licht's inductive construction \cite{li19} as expressed here in Corollary~\ref{cor:induct}, except that we additionally include as base cases that $\cP_r\Lambda^k(T^3)$ always has an $S_4$-invariant basis up to sign by Corollary~\ref{cor:3dbasis}, from which it also follows that $\oP_r^-\Lambda^k(T^3)$ always has an $S_4$-invariant basis up to sign by Proposition~\ref{prop:popiso}.

  We begin with $\cP_r^-\Lambda^2(T^3)$. By Corollary~\ref{cor:induct}, we can construct an invariant basis for $\cP_r^-\Lambda^2(T^3)$ if we can construct an invariant basis for $\oP_r^-\Lambda^2(T^3)$ and $\oP_r^-\Lambda^2(T^2)$. As just discussed, we can always construct an invariant basis for $\oP_r^-\Lambda^2(T^3)$, and we can always construct an invariant basis for $\oP_r^-\Lambda^2(T^2)$ because it is a space of top-level forms. Thus, $\cP_r^-\Lambda^2(T^3)$ always has an invariant basis, from which we can conclude that so does $\oP_r\Lambda^1(T^3)$ by Proposition~\ref{prop:popiso}.

  Next, consider $\cP_r^-\Lambda^1(T^3)$. By Corollary~\ref{cor:induct}, we can construct an invariant basis for $\cP_r^-\Lambda^1(T^3)$ if we can construct an invariant basis for $\oP_r^-\Lambda^1(T^3)$, $\oP_r^-\Lambda^1(T^2)$ and $\oP_r^-\Lambda^1(T^1)$. As we discussed at the beginning of the proof, doing so is always possible for $\oP_r^-\Lambda^1(T^3)$. By Theorem~\ref{thm:2d}, we can construct such a basis for $\oP_r^-\Lambda^1(T^2)$ if $r\notin3\mathbb N_0+2$. Finally, we can always construct an invariant basis for $\oP_r^-\Lambda^1(T^1)$. We conclude that an invariant basis for $\cP_r^-\Lambda^1(T^3)$ exists if $r\notin3\mathbb N_0+2$. In this case, Proposition~\ref{prop:popiso} gives an isomorphism up to sign between $\cP_r^-\Lambda^1(T^3)$ and $\oP_{r+1}\Lambda^2(T^3)$, so we conclude that an invariant basis for $\oP_r\Lambda^2(T^3)$ exists if $r\notin3\mathbb N_0+3$.
\end{proof}

\begin{proof}[Proof, part 2: Obstructing invariant bases]
  To prove the converse, we follow a similar strategy to the proof of Theorem~\ref{thm:2d}. Assume that $\cP_r^-\Lambda^1(T^3)$ has an $S_4$-invariant basis up to sign. Then it also has a $\Z/3$-invariant basis up to sign, where, following earlier notation, we take $\Z/3\subset S_3\subset S_4$ to be the subgroup consisting of permutations of $\{0,1,2,3\}$ that fix $3$ and cyclically permute $0$, $1$, and $2$. By Corollary~\ref{cor:geodecomp}, we have the following isomorphism of $\Z/3$-representations.
  \begin{equation*}
    \begin{split}
      \cP_r^-\Lambda^1(T^3)&\cong\left(\oP_r^-\Lambda^1(F_{12})\oplus\oP_r^-\Lambda^1(F_{20})\oplus\oP_r^-\Lambda^1(F_{01})\right)\\
      &\quad\oplus\left(\oP_r^-\Lambda^1(F_{03})\oplus\oP_r^-\Lambda^1(F_{13})\oplus\oP_r^-\Lambda^1(F_{23})\right)\\
      &\quad\oplus\left(\oP_r^-\Lambda^1(F_{123})\oplus\oP_r^-\Lambda^1(F_{203})\oplus\oP_r^-\Lambda^1(F_{013})\right)\\
      &\quad\oplus\oP_r^-\Lambda^1(T^2)\\
      &\quad\oplus\oP_r^-\Lambda^1(T^3).
      \end{split}
    \end{equation*}
    Here, as before, the subscript notation $F_{\cdots}$ denotes the face spanned by the given vertices, and we take $T^2$ to mean $F_{012}$. Each of the first three lines satisfies the hypotheses of Proposition~\ref{prop:3s}, so the first three lines are sums of copies of the representation $\mb3$. Next, using Proposition~\ref{prop:popiso}, we have, as $\Z/3$-representations, the isomorphisms
    \begin{equation*}
      \oP_r^-\Lambda^1(T^3)\cong\cP_{r-3}\Lambda^2(T^3)\cong\cP_{r-3}\Lambda^0(T^3)\otimes\cP_0\Lambda^2(T^3)\cong\cP_{r-3}\Lambda^0(T^3)\otimes\mb3,
    \end{equation*}
    which is also a sum of copies of the representation $\mb3$ by Corollary~\ref{cor:tensor3}. We conclude by Corollary~\ref{cor:obstruction} that $\cP_r^-\Lambda^1(T^3)$ has a $\Z/3$-invariant basis up to sign if and only if $\oP_r^-\Lambda^1(T^2)$ does, which we showed in the proof of Theorem~\ref{thm:2d} happens precisely when $r\notin3\mathbb N_0+2$.

    As above, Proposition~\ref{prop:popiso} then implies that $\oP_r\Lambda^2(T^3)$ has a $\Z/3$-invariant basis up to sign if and only if $r\notin3\mathbb N_0+3$; in particular, it cannot have an $S_4$-invariant basis up to sign if $r\in3\mathbb N_0+3$.
  \end{proof}

  Licht \cite{li19} also asks whether these bases can be geometrically decomposed. All of the bases construct by Licht can be, and, additionally, Licht shows that if there are bases not found by his construction, then at least one of them \emph{cannot} be geometrically decomposed. To have this discussion, we first define a notion of a geometrically decomposable basis.
  
  \begin{definition}
    Let $V$ be a space of $k$-forms, and let $\mathcal D$ be the geometric decomposition map as in Definition~\ref{def:geodecomp}. We say that a basis for $V$ can be \emph{geometrically decomposed} if for every $\xi$ in the basis, $\mathcal D(\xi)$ lies in one of the direct summands $\oV(F)$ for a face $F$ of $T^n$.
  \end{definition}

  If the geometric decomposition map is surjective, then such a basis for $V$ yields bases for each $\oV(F)$, and, in particular, for each $\oV(T^d)$. Moreover, if this basis is $S_{n+1}$-invariant up to sign, then the corresponding bases of $\oV(T^d)$ are $S_{d+1}$-invariant up to sign. Thus, asking whether an $S_{n+1}$-invariant basis can be geometrically decomposed amounts to asking whether the converse to Corollary~\ref{cor:induct} holds. In light of Theorems~\ref{thm:2d} and~\ref{thm:3d}, we can quickly answer this question.

  \begin{corollary}
    Let $r\ge1$. The following spaces have $S_4$-invariant bases up to sign that can be geometrically decomposed if and only if the corresponding condition holds.
    \begin{equation*}
      \begin{array}{ccl}
        \cP_r\Lambda^1(T^3)& \text{if and only if}&r\notin3\mathbb N_0+3,\\
        \cP_r^-\Lambda^1(T^3)& \quad\text{if and only if}\quad &r\notin 3\mathbb N_0+2,\\
        \cP_r\Lambda^2(T^3)& \text{if and only if}&r\notin3\mathbb N_0+3,\\
        \cP_r^-\Lambda^2(T^3)& \text{always}.\\
      \end{array}
    \end{equation*}
  \end{corollary}

  \begin{proof}
    For $\cP_r\Lambda^1(T^3)$, we must check whether $\oP_r\Lambda^1(T^3)$, $\oP_r\Lambda^1(T^2)$, and $\oP_r\Lambda^1(T^1)$ all have invariant bases up to sign, which happens if and only if $r\notin3\mathbb N_0+3$.

    For $\cP_r^-\Lambda^1(T^3)$, we must check whether $\oP_r^-\Lambda^1(T^3)$, $\oP_r^-\Lambda^1(T^2)$, and $\oP_r^-\Lambda^1(T^1)$ all have invariant bases up to sign, which happens if and only if $r\notin3\mathbb N_0+2$.

    For $\cP_r\Lambda^2(T^3)$, we must check whether $\oP_r\Lambda^2(T^3)$ and $\oP_r\Lambda^2(T^2)$ have invariant bases up to sign, which happens if and only if $r\notin3\mathbb N_0+3$.

    For $\cP_r^-\Lambda^2(T^3)$, we must check whether $\oP_r^-\Lambda^2(T^3)$ and $\oP_r^-\Lambda^2(T^2)$ have invariant bases up to sign, which is always the case.
  \end{proof}

  \begin{remark}
    As in Licht's work \cite{li19}, the inductive argument gives an explicit way of constructing invariant bases. The extension maps defined in \cite{afw09} and used by Licht \cite{li19} give an explicit way of combining invariant bases for the summands of the geometric decomposition to give an invariant basis for the full space, and the duality isomorphism is an explicit map that allows one to take an invariant basis for a finite element space and obtain an invariant basis for the corresponding space with vanishing trace, and vice versa. The new ingredient in this paper for constructing invariant bases is the invariant basis for $\cP_r\Lambda^k(T^3)$, which is once again given explicitly by tensoring the monomial basis for $\cP_r\Lambda^0(T^3)$ with the explicit bases for $\cP_0\Lambda^1(T^3)$ and $\cP_0\Lambda^2(T^3)$ described geometrically in Subsection~\ref{subsec:invariant} and written algebraically in \cite{li19}.
  \end{remark}

  \begin{remark}
    While the special situation in dimension three allowed us to take some shortcuts, the ideas in this paper can be used in any dimension to determine exactly which finite element spaces have invariant bases. The geometric decomposition in Corollary~\ref{cor:geodecomp} along with the duality isomorphism in Proposition~\ref{prop:popiso} allow us to recursively decompose the $\cP_r\Lambda^k(T^n)$ and $\cP_r^-\Lambda^k(T^n)$ spaces into irreducible $S_{n+1}$-representations. In turn, the ideas of Proposition~\ref{prop:obstruction} can be used for any group, including $S_{n+1}$, to determine which representations have invariant bases. Specifically, one would compute $\Ind_{H}^{S_{n+1}}L$ for all subgroups $H$ of $S_{n+1}$ and all one-dimensional representations $L$ of $H$. Writing each of these as a sum of irreducible representations, we obtain a generating set for the representations with invariant bases, allowing us to use a given representation's decomposition into irreducibles to determine whether or not it has an invariant basis.
  \end{remark}

\appendix

\section{Consistent extension operators}\label{sec:extension}
In this appendix, we show that the extension operators discussed in \cite{afw09} yield valid choices of complements $W_d$ as discussed in Section~\ref{sec:geodecomp}, and so, in this context, the notion of geometric decomposition in Section~\ref{sec:geodecomp} matches the one in \cite{afw09}. However, we also show that our notion of geometric decomposition is more general (and hence weaker). For our results, our simpler and more general setup suffices, but we expect that other work in finite element exterior calculus may require the more stringent requirements on the extension maps in \cite{afw09}.

In our notation, we take the following definition from \cite[Section 4]{afw09}.

\begin{definition}
  For faces $K\subseteq F\subseteq T^n$, an operator $E_{K, F}\colon V(K)\to V(F)$ is an \emph{extension operator} if $\tr_{F, K}E_{K, F}\,\alpha=\alpha$ for all $\alpha\in V(K)$.

  A family of such extension operators $\{E_{F,F'}\mid F\subseteq F'\subseteq T^n\}$ is \emph{consistent} if the diagram
  \begin{equation*}
    \begin{tikzcd}
      V(F)\arrow[r, "E_{F, H}"]\arrow[d, "\tr_{F, K}"]&V(H)\arrow[d, "\tr_{H, F'}"]\\
      V(K)\arrow[r, "E_{K, F'}"]&V(F')
    \end{tikzcd}
  \end{equation*}
  commutes for any subfaces $F$ and $F'$ of a face $H$, where $K=F\cap F'$.
\end{definition}

Given a consistent family of extension operators, we can construct corresponding complements $W_d$ as in Definition~\ref{def:W}.

\begin{proposition}\label{prop:extensiontoW}
  Given a consistent family of extension operators, let
  \begin{equation*}
    W_d:=\bigoplus_{\substack{F\subseteq T^n\\\dim F=d}}E_{F, T^n}\oV(F).
  \end{equation*}
  Then $V_d=W_d\oplus V_{d+1}$, where the $V_d$ are defined in Definition~\ref{def:Vd}.
\end{proposition}

\begin{proof}
  First, the direct sum notation in the definition of $W_d$ implicitly assumes that $E_{F, T^n}\oV(F)\cap E_{F', T^n}\oV(F')=0$; this is a consequence of \cite[Equation~4.9]{afw09}. Moreover, by \cite[Lemma~4.1]{afw09}, if $F'$ and $F$ are $d$-dimensional faces and $F'\neq F$, then the composition $\tr_{T^n, F'}E_{F, T^n}$ is the zero map on $\oV(F)$. On the other hand, if $F'=F$, then the composition $\tr_{T^n, F'}E_{F, T^n}$ is the identity map. Consequently, for any $\beta\in W_d$ written as
  \begin{equation*}
    \beta=\bigoplus_{\substack{F\subseteq T^n\\\dim F=d}}E_{F, T^n}\beta_F,
  \end{equation*}
  where $\beta_F\in\oV(F)$, we have
  \begin{equation}\label{eq:Wproject}
    \tr_{T^n, F'}\beta=\beta_{F'}
  \end{equation}
  for any $d$-dimensional face $F'$. From equation~\eqref{eq:Wproject}, it is easy to see that $W_d\cap V_{d+1}=0$. Indeed, with $\beta$ as above, if $\beta\in V_{d+1}$ then by definition $\tr_{T^n, F'}\beta=0$ for any $d$-dimensional face $F'$, so $\beta_{F'}=0$ for all $d$-dimensional $F'$, and so $\beta=0$.
  
  Now, let $\alpha\in V_d$; we aim to write $\alpha=\beta+\gamma$ where $\beta\in W_d$ and $\gamma\in V_{d+1}$. By definition of $V_d$, we know that $\alpha$ vanishes on any $(d-1)$-dimensional face of $T^n$, and hence $\tr_{T^n, F}\alpha\in\oV(F)$ for any $d$-dimensional face $F$ of $T^n$. Thus, we can let
  \begin{equation*}
    \beta=\bigoplus_{\substack{F\subseteq T^n\\\dim F=d}}E_{F, T^n}\tr_{T^n, F}\alpha\in W_d.
  \end{equation*}
  By equation~\eqref{eq:Wproject}, if $F'$ is a $d$-dimensional face, then
  \begin{equation*}
    \tr_{T^n, F'}\beta=\tr_{T^n, F'}\alpha.
  \end{equation*}
  Thus, if we let $\gamma=\alpha-\beta$, then $\tr_{T^n, F'}\gamma=0$ for any $d$-dimensional face $F'$, so $\gamma\in V_{d+1}$ by definition.
\end{proof}

\begin{remark}
  The decomposition $V=W_0\oplus\dotsb\oplus W_n$ in Proposition~\ref{prop:Wdecomp} then exactly yields the decomposition $V=\bigoplus_{F\subseteq T^n}E_{F, T^n}\oV(F)$ given in \cite[Equation~4.9]{afw09}. Moreover, by equation~\eqref{eq:Wproject}, the geometric decomposition map $\mathcal D$ in Definition~\ref{def:geodecomp} is just the natural map $\bigoplus_{F\subseteq T^n}E_{F, T^n}\oV(F)\to\bigoplus_{F\subseteq T^n}\oV(F)$.
\end{remark}

\begin{remark}\label{rem:converse}
  Only some of the converse holds. Given complements $W_d$, we obtain an injective geometric decomposition map $\mathcal D$ as in Definition~\ref{def:geodecomp}. If this map is furthermore an isomorphism, then we have extension maps $E_{F, T^n}\colon\oV(F)\to W_d\subseteq V$ as discussed in Remark~\ref{rem:extension}. As discussed in Corollary~\ref{cor:extensionproperties}, these extension maps enjoy several properties. In particular, if $\alpha\in E_{F, T^n}\oV(F)$, then its restriction $\tr_{T^n, F'}\alpha$ to any other $d$-dimensional face $F'\neq F$ is zero. However, \cite[Lemma 4.1]{afw09} makes a stronger claim, that $\tr_{T^n, F'}\alpha$ is zero for any face $F'$ (of any dimension) not containing $F$; this may fail if $F'$ has higher dimension than $F$.

  Indeed, in the case where $V=\cP_2\Lambda^0(T^2)$ discussed in Example~\ref{eg:P2}, the choice of of $W_0$ is quite flexible. The natural choice is $W_0=\Span\{\lambda_0^2,\lambda_1^2,\lambda_2^2\}$, but, even if we require invariance with respect to the $S_3$ symmetry, nothing stops us from choosing
  \begin{equation*}
    W_0=\Span\{\lambda_0^2+17\lambda_1\lambda_2,\lambda_1^2+17\lambda_2\lambda_0,\lambda_2^2+17\lambda_0\lambda_1\}.
  \end{equation*}
  With this latter choice, if $F$ is the vertex $0$ and $F'$ is the edge $12$, one can check that $E_{F, T^2}(1)=\lambda_0^2+17\lambda_1\lambda_2$ and $\tr_{T^2, F'}(\lambda_0^2+17\lambda_1\lambda_2)=17\lambda_1\lambda_2\neq0$.
\end{remark}
  
\subsection*{$S_{n+1}$-invariance}
In the context where $V$ is $S_{n+1}$-invariant, we want $W_d$ to be $S_{n+1}$-invariant as well. In light of the construction in Proposition~\ref{prop:extensiontoW}, it suffices to require that the extension operators respect this action, in the sense that the diagram
\begin{equation*}
  \begin{tikzcd}
    V\arrow[r, "S_\pi^*"]&V\\
    V(F')\arrow[r, "S_\pi^*"]\arrow[u,"E_{F', T^n}", swap]&V(F)\arrow[u,"E_{F, T^n}", swap]
  \end{tikzcd}
\end{equation*}
commutes, where $F'=S_\pi F$, analogously to diagram~\eqref{eq:Vdiagram} in Subsection~\ref{subsec:representations}. There is no reason to expect this property to be true in general, but it holds for extension operators defined by ``natural'' properties, including the extension operators constructed in \cite{afw09} for the $\cP_r\Lambda^k(T^n)$ and $\cP_r^-\Lambda^k(T^n)$ spaces. We roughly sketch the arguments below; see also \cite[Section~7]{li19}.

There are actually two sets of extension operators constructed in \cite{afw09}. The first, denoted with the symbols $F^{k, r}_{F, T^n}$ and $F^{k, r,-}_{F, T^n}$ in \cite[Section 5]{afw09}, is defined in terms of degrees of freedom: the degrees of freedom of $F^{k, r, (-)}_{F, T^n}\alpha$ must match those of $\alpha$ on all subfaces $K\subseteq F$ and be zero on all other faces of $T^n$; this uniquely determines the extension. One then checks that if $\alpha'\in\cP_r^{(-)}\Lambda^k(F')$, then the condition that the degrees of freedom of $F^{k, r, (-)}_{F', T^n}\alpha'$ match those of $\alpha'$ implies that the degrees of freedom of $S_\pi^*F^{k, r, (-)}_{F', T^n}\alpha'$ match those of $S_\pi^*\alpha'$, and likewise for the degrees of freedom required to be zero. (Intuitively, rotating $T^n$ ``rotates'' the degrees of freedom.) Since this condition uniquely determines the extension, we conclude that $S_\pi^*F^{k, r, (-)}_{F', T^n}\alpha'$ is indeed the extension of $S_\pi^*\alpha'$, as desired.

The reasoning is similar for the extension operators $E^{k, r}_{F, T^n}$ and $E^{k, r, -}_{F, T^n}$ defined in \cite[Sections 7 and 8]{afw09}. In \cite[Theorems 7.4 and 8.4]{afw09}, the authors show that $E^{k, r, (-)}_{F, T^n}\alpha$ is the unique extension of $\alpha$ that satisfies a certain vanishing condition on the face $F^*$ that is ``opposite'' $F$. As above, one checks that if $\alpha'\in\cP_r^{(-)}\Lambda^k(F')$, then this vanishing condition for $E^{k, r, (-)}_{F', T^n}\alpha'$ on $F'^*$ implies the corresponding vanishing condition for $S_\pi^*E^{k, r, (-)}_{F', T^n}\alpha'$ with respect to $F^*$. Since this condition uniquely determines the extension, we conclude that $S_\pi^*E^{k, r, (-)}_{F', T^n}\alpha'$ is indeed the extension of $S_\pi^*\alpha$.

\bibliographystyle{siam}
\bibliography{fem}

\begin{thebibliography}{10}

\bibitem{a13}
{\sc D.~N. Arnold}, {\em Spaces of finite element differential forms}, in
  Analysis and Numerics of Partial Differential Equations, vol.~4 of Springer
  INdAM Ser., Springer, Milan, 2013, pp.~117--140.

\bibitem{afw06}
{\sc D.~N. Arnold, R.~S. Falk, and R.~Winther}, {\em Finite element exterior
  calculus, homological techniques, and applications}, Acta Numer., 15 (2006),
  pp.~1--155.

\bibitem{afw09}
\leavevmode\vrule height 2pt depth -1.6pt width 23pt, {\em Geometric
  decompositions and local bases for spaces of finite element differential
  forms}, Comput. Methods Appl. Mech. Engrg., 198 (2009), pp.~1660--1672.

\bibitem{afw10}
\leavevmode\vrule height 2pt depth -1.6pt width 23pt, {\em Finite element
  exterior calculus: from {H}odge theory to numerical stability}, Bull. Amer.
  Math. Soc. (N.S.), 47 (2010), pp.~281--354.

\bibitem{bk21feec}
{\sc Y.~Berchenko-Kogan}, {\em Duality in finite element exterior calculus and
  {H}odge duality on the sphere}, Found. Comput. Math., 21 (2021),
  pp.~1153--1180.
\newblock \url{https://rdcu.be/cdSpS}.

\bibitem{bdm85}
{\sc F.~Brezzi, J.~Douglas, Jr., and L.~D. Marini}, {\em Two families of mixed
  finite elements for second order elliptic problems}, Numer. Math., 47 (1985),
  pp.~217--235.

\bibitem{cure62}
{\sc C.~W. Curtis and I.~Reiner}, {\em Representation Theory of Finite Groups
  and Associative Algebras}, Pure and Applied Mathematics, Vol. XI,
  Interscience Publishers, 1962.

\bibitem{fuha91}
{\sc W.~Fulton and J.~Harris}, {\em Representation Theory: A First Course},
  vol.~129 of Graduate Texts in Mathematics, Springer-Verlag, New York, 1991.

\bibitem{hi03}
{\sc A.~N. Hirani}, {\em Discrete exterior calculus}, 2003.
\newblock Thesis (Ph.D.)--California Institute of Technology.

\bibitem{li19}
{\sc M.~W. Licht}, {\em Symmetry and invariant bases in finite element exterior
  calculus}, 2019.
\newblock \url{https://arxiv.org/abs/1912.11002}.

\bibitem{li18}
{\sc M.~W. Licht}, {\em On basis constructions in finite element exterior
  calculus}, Adv. Comput. Math., 48 (2022), pp.~Paper No. 14, 36.

\bibitem{n80}
{\sc J.-C. N{\'e}d{\'e}lec}, {\em Mixed finite elements in {$\mathbb{R}^{3}$}},
  Numer. Math., 35 (1980), pp.~315--341.

\bibitem{n86}
\leavevmode\vrule height 2pt depth -1.6pt width 23pt, {\em A new family of
  mixed finite elements in {$\mathbb{R}^{3}$}}, Numer. Math., 50 (1986),
  pp.~57--81.

\bibitem{rt77}
{\sc P.-A. Raviart and J.~M. Thomas}, {\em A mixed finite element method for
  2nd order elliptic problems},  (1977), pp.~292--315. Lecture Notes in Math.,
  Vol. 606.

\bibitem{se77}
{\sc J.-P. Serre}, {\em Linear Representations of Finite Groups}, Graduate
  Texts in Mathematics, Vol. 42, Springer-Verlag, New York-Heidelberg, 1977.
\newblock Translated from the second French edition by Leonard L. Scott.

\bibitem{wh57}
{\sc H.~Whitney}, {\em Geometric Integration Theory}, Princeton University
  Press, Princeton, N. J., 1957.

\end{thebibliography}
\end{document}